%% file: Inexact-Fista.tex
\documentclass{siamart0516}

\usepackage{color}
\usepackage{float}
\usepackage{textcomp}
\usepackage{amsmath}
\usepackage{amssymb}
\usepackage{algorithmic}
\usepackage{changes}
\usepackage{soul}
\usepackage{mathtools}

\usepackage{booktabs}
\usepackage{tabularx} 
\usepackage{jlcode}

\DeclareMathOperator*{\argmin}{argmin}

\newcommand{\inner}[2]{\langle{#1},{#2}\rangle}

\newcommand{\NN}{\ensuremath{\mathbb N}}

\newcommand{\dom}{\ensuremath{\operatorname{dom}}}
\newcommand{\scal}[2]{\left\langle{#1},{#2}\right\rangle}
\newcommand{\prox}{\ensuremath{\operatorname{Prox}}}
\newcommand{\Id}{\ensuremath{\operatorname{Id}}}

\def\NN{{\mathbb{N}}}

\DeclarePairedDelimiter\rb{ ( }{ ) }
\def\ve{\varepsilon}
\newtheorem{algorithm}{Algorithm}
\newtheorem{ex}{Rule}
\newtheorem{remark}{Remark}

\newcommand{\bi}{\begin{itemize}}
\newcommand{\ei}{\end{itemize}}
\newcommand{\ba}{\begin{array}}
\newcommand{\ea}{\end{array}}

\input{Teste_shared}

\ifpdf
\hypersetup{
  pdftitle={Regularized ADMM},
  pdfauthor={Max}
}
\fi

\begin{document}

\maketitle

\begin{abstract} 
One of the most popular and important first-order iterations that provides optimal complexity of the classical proximal gradient method (PGM) is the ``Fast Iterative Shrinkage/Thresholding Algorithm'' (FISTA). In this paper, two inexact versions of FISTA for minimizing the sum of two convex functions are studied. The proposed schemes inexactly solve their subproblems by using relative error criteria instead of exogenous and diminishing error rules. When the evaluation of the proximal operator is difficult, inexact versions of FISTA are necessary and the relative error rules proposed here may have certain advantages over previous error rules. The same optimal convergence rate of FISTA is recovered for both proposed schemes. Some numerical experiments are reported to illustrate the numerical behavior of the new approaches.
\end{abstract}



\section{Introduction}\label{sec:int}
Throughout this paper, we write $p \coloneqq q$ to indicate that $p$ is defined to be equal to $q$. The nonnegative (positive) numbers will be denoted by $\mathbb{R}_+$ ($\mathbb{R}_{++}$). Moreover, $\mathbb{E}$ denotes a finite-dimensional real vector space, which is equipped with the inner product $\scal{\,\cdot\,}{\cdot\,}$ and its induced norm $\|\cdot\|$.

Consider the following problem
\begin{equation}\label{pr1}
\min_{x \in\mathbb{E}}F(x)\coloneqq f(x)+g(x),
\end{equation}
where $f\colon \mathbb{E}\to \mathbb{R}$ is a differentiable convex function whose gradient is $L$-Lipschitz continuous and $g\colon \mathbb{E}\to \overline{\mathbb{R}}\coloneqq \mathbb{R}\cup \{+\infty\}$ is a lower semicontinuous (lsc) convex function that is not necessarily differentiable.
We denote the optimal value of \cref{pr1} by $F^*$, the set of optimal solutions of \cref{pr1} by $S_*$, and we assume that $F^* \in \mathbb{R}$ and that $S_*$ is nonempty; thus we have $F(x_*) = F^*$, for all $x_* \in S_*$.
It is well-known that \cref{pr1} contains a wide class of problems arising in applications from science and engineering, including machine learning, compressed sensing, and image processing. There are important examples of this problem such as using $\ell_{1}$--regularization to obtain sparse solutions with applications in signal recovery and signal processing problems \cite{BTe,HYZ,Tr}, the nearest correlation matrix problem \cite{Borsdorf2010a,Defeng-S-2012,Qi2006a}, and regularized inverse problems with atomic norms \cite{Villa-Salzo-Luca-Verri-2013}.

A plethora of methods has been proposed for solving the aforementioned optimization problem. One of the most studied approaches is the proximal gradient method (PGM) which is a first-order splitting iteration that has been intensively investigated in the literature; see, for instance, \cite{FISTA, YN2,YN}. PGM iterates by performing a gradient step based on $f$ followed by the evaluation of the {\em proximal} (or $\prox$) {\em operator} of $g$,
which is defined as $\prox_g\coloneqq (\Id+\partial g)^{-1}$ where
$$\partial g(x)\coloneqq
\bigl\{
u \in \mathbb{E}
\bigm\vert
g(y)\ge g(x)+ \langle u, y-x\rangle,
\forall y \in \mathbb{E}
\bigr\}$$
is the subdifferential of $g$ at $x\in \mathbb{E}$ and $\Id$ is the identity operator.
 It is well-known that the sequence $(x_k)_{k\in \mathbb{N}}$ generated by PGM has a complexity rate of $\mathcal{O}(\rho^{-1})$ to obtain a $\rho$--approximate solution of \cref{pr1} (that is, a solution $x_k$ satisfying $F(x_k)-F^*\le \rho$), or equivalently we can say that $F(x_k)-F^*=\mathcal{O}(k^{-1})$; see, for instance, \cite{FISTA, YN2, YN}. In addition, it is possible to accelerate the proximal gradient method in order to
achieve the optimal $\mathcal{O}(k^{-2})$ convergence rate by adding an extrapolation step. This
scheme, which improved the complexity of the gradient method for minimizing smooth convex functions, was first introduced by Nesterov in 1983 \cite{Nesterov1983} and further extended to constrained problems in 1988 \cite{Nesterov1988,Nesterov2005}. In the spirit of the work of \cite{Nesterov1983}, Nesterov \cite{Nesterov2013} (appeared online in 2007 but published in 2013) and Beck--Teboulle \cite{FISTA} extended Nesterov's classical iteration to minimizing composite nonsmooth functions.

In this paper, we propose a modification of the ``Fast Iterative Shrinkage/Thresh\-olding Algorithm'' (FISTA) of \cite{FISTA}. FISTA  is described  as follows.
\medskip
\begin{center}\fbox{\begin{minipage}[b]{0.981\textwidth}
\begin{algorithm}[FISTA]
Let $x_{0} \in \mathbb{E}$, and $L>0$ be the Lipschitz constant of $\nabla f$. Set $y_{1}\coloneqq x_{0}$, $t_1 \coloneqq 1$, and iterate
	\begin{align}\label{xk-FISTA}
	x_{k} & \coloneqq \prox_{\frac{1}{L}g}\left(y_k-\frac{1}{L}\nabla f(y_k)\right), \\
	t_{k+1} & \coloneqq \frac{1+\sqrt{1+4t_{k}^{2}} }{2},\label{tk-FISTA}\\\label{yk-FISTA}
	y_{k+1} & \displaystyle
	\coloneqq
	x_{k} + \frac{t_{k}-1}{t_{k+1}}
    \rb{x_{k} - x_{k-1}}.
	\end{align}
\end{algorithm}
\end{minipage}}\end{center}
\medskip
Note that if the update \cref{tk-FISTA} is ignored and $t_k= 1$ for all $k\in \mathbb{N}$, FISTA becomes the (unaccelerated)
PGM mentioned before.
There are two very popular choices for the sequence $(t_k)_{k\in\mathbb{N}}$ \cite{FISTA,Nesterov2013} but
several different updates are possible for $t_k$ that also achieve the optimal acceleration; see, for instance,
\cite{Attouch-fast-MPB-2018,Attouch-rate-2016,Chambolle-Dossal-15,Su-Boyd-Candes-2016}. Convergence and complexity results of the sequence generated by FISTA under a suitable tuning of $(t_k)_{k\in\mathbb{N}}$ related to the update \cref{tk-FISTA} can be found in \cite{Attouch-Cabot-HAL2017,Attouch-rate-2016,YN,Chambolle-Dossal-15}. Many accelerated versions have been proposed in the literature for accelerating the PGM for solving \cref{pr1}. The relaxed case was considered in \cite{Aujol-Dossal-2015} and error-tolerant versions were studied in \cite{Attouch-JOTA-18,Attouch-fast-MPB-2018}.
In addition, for results concerning the rate of convergence of function values of \cref{pr1} with or without minimizers, see \cite{B-Bu-2019,Su-Boyd-Candes-2016}.

FISTA (and in particular PGM) is an effective and simple choice for solving large scale problems when the $\prox$ operator has a closed-form or there exists an efficient way to evaluate it. Frequently, it could be computationally expensive to evaluate the $\prox$ operator at any point with high accuracy \cite{FB-2014}. The theory of convergence for the (accelerated) PGM assumes that the $\prox$ operator can be evaluated at any point; that is, the regularized minimization problem
\[
\min_{x\in\mathbb{E}} \left\{ g(x)+\frac{1}{2\gamma}\|x-z\|^2\right\}, \quad \gamma>0,
\]
can be solved for any $z\in\mathbb{E}$. The unique solution of the above problem is actually the $\prox$ operator of $\gamma g$ at $z$, which is the function $\prox_{\gamma g}\colon\mathbb{E}\to \dom g$ defined by
\[
\prox_{\gamma g}(z)\coloneqq \argmin_{x\in\mathbb{E}} \left\{g(x)+\frac{1}{2\gamma}\|x-z\|^2\right\}.
\]
This function satisfies the following necessary and sufficient optimality condition:
\[
 \frac{1}{\gamma}\left(z-\prox_{\gamma g}(z)\right)\in \partial g(\prox_{\gamma g}(z)).
\]
Therefore, to run FISTA we must compute $x_k$ by solving the subproblem
\begin{equation}\label{xk-FISTA-S}
\min_{x\in \mathbb{E}}\left\{ g(x)+\frac L2\left\|x-\left(y_k-\frac{1}{L}\nabla f(y_k)\right)\right\|^2\right\}.
\end{equation} That is, we must find
the point $x_k$ that satisfies
$$
0 \in \partial g(x_k)+L(x_k-y_k) +\nabla f(y_k).
$$
A natural question is: What happens if the solution of \cref{xk-FISTA-S} can not be easily computed? Often in practice in this case the evaluation of the proximal operator is done approximately. However, to guarantee an optimal complexity rate, it is required that the nonnegative sequence of error tolerances be summable. As was shown in \cite{Defeng-S-2012,Villa-Salzo-Luca-Verri-2013}, with a summable sequence of error tolerances for these approximate solutions, the optimal complexity rate $\mathcal{O}(k^{-2})$ of FISTA is recovered.

The two works \cite{Defeng-S-2012,Villa-Salzo-Luca-Verri-2013} appeared simultaneously around 2013 and proposed inexact variations of FISTA with summable error tolerances for computing the $\ve$--approximate solutions of subproblem \cref{xk-FISTA-S}. In \cite{Villa-Salzo-Luca-Verri-2013}, given a nonnegative sequence $(\ve_k)_{k\in \mathbb{N}}$, iterates $\tilde x_k$ are generated such that
\begin{equation}\label{inexact-Sum}
0 \in \partial_{\ve_k} g(\tilde x_k)+L(\tilde x_k-y_k) + \nabla f(y_k),
\end{equation}
where
$$\partial_\ve g(x)\coloneqq
\bigl\{
u \in \mathbb{E}
\bigm\vert
g(y) \ge g(x) + \inner{u}{y-x} - \ve, \forall y \in \mathbb{E}
\bigr\}$$
is an enlargement of $\partial g$.
On the other hand, the version in \cite{Defeng-S-2012} allows the approximate solution $\tilde x_k$ of subproblem \cref{xk-FISTA-S} such that
\begin{gather}\label{eq:eid34}
F(\tilde x_k)\leq g(\tilde x_k)+f(y_k)+\inner{\nabla f(y_k)}{\tilde x_k - y_k}+\inner{\tilde x_k - y_k}{H_k(\tilde x_k - y_k)}+ \frac{\xi_k}{2 t_k^2},\\\nonumber
v_k \in \partial_{ \frac{\xi_k}{2 t_k^2}}g(\tilde x_k) + H_k(\tilde x_k - y_k) + \nabla f(y_k), \quad
\|H_k^{-1/2}v_k\| \le \frac{\delta_k}{\sqrt{2} t_k},
\end{gather}
where   $(\delta_k)_{k\in \mathbb{N}}$ and $(\xi_k)_{k\in \mathbb{N}}$ are summable sequences of nonnegative numbers, and $H_k$ is  a self-adjoint positive definite operator. If $H_k=L\Id$, then inequality \eqref{eq:eid34} is trivially satisfied.
In both of these approaches, the summability assumption may require us to find $\tilde x_k$ to a  level of accuracy that is higher than necessary.

In the spirit of \cite{Complexity-HPE,HPE2}, we propose two inexact versions with {\em relative} error rules for solving the main subproblem of FISTA. The advantages over the inexact methods given in \cite{Defeng-S-2012,Villa-Salzo-Luca-Verri-2013} are the following:
\begin{itemize}

\item[\bf (a)] The proposed relative error rules have no summability assumption and the error tolerances naturally depend on the generated iterates. Our first proposed method is a generalization of FISTA and our second proposed method is related to the extra-step acceleration method proposed in \cite{acc-HPE}.

\item[\bf (b)] We recover the optimal iteration convergence rate in terms of the objective function value for both proposed inexact methods.
Moreover, for a given  tolerance $\rho>0$,  we also study iteration-complexity bounds  for the proposed algorithms in order to obtain a $\rho$--approximate solution $x$ of the inclusion $0\in \partial F(x)$ with residual $(r,\ve)$, i.e.,
\[
r\in \partial_{\ve} F(x), \quad \max\{\|r\|,\ve\}\leq \rho.
\]
Since $0\in \partial F(x_*) $, for all $x_* \in S_*$, the latter condition can be interpreted as an optimality measure for $x$.
\end{itemize}

The presentation of this paper is as follows. Definitions, basic facts and auxiliary results are presented in \cref{pre}. Our inexact criteria with relative error rules are presented in \cref{Inexact-Rules}. In \cref{PFA,IEAM} we present the inexact algorithms and their convergence rates. Some  numerical experiments for the proposed schemes are reported in \cref{NumSec}. Finally, some concluding remarks are given in \cref{Finalremarks}.

\section{Definitions and auxiliary results}\label{pre}

Let $h \colon \mathbb{E} \to \overline{\mathbb{R}}$ be a proper, convex, and lower
semicontinuous (\emph{l.s.c.}) function. We denote the
domain of $h$ by $\dom h \coloneqq \{x\in \mathbb{E} \mid h(x)<+\infty\}$. Recall that the proximal operator $\prox_h \colon \mathbb{E}\to \dom g$ is defined by $\prox_h(x):=(\Id+\partial h)^{-1}(x)$. It is well-known that the proximal operator is single-valued with full domain, is continuous, and has many other attractive properties. In particular, the proximal operator is firmly nonexpansive:
\[
\left\|\prox_{h}(x)-\prox_{h}(y)\right\|^2\le \|x-y\|^2-\left\|(x-\prox_{h}(x))-(y-\prox_{h}(y))\right\|^2,
\]
for all $x,y\in\mathbb{E}$. Moreover,
\[
0\in \partial g(\prox_{\gamma h}(x))+\frac{1}{\gamma}\left(\prox_{\gamma h}(x)-x\right),\quad x\in\mathbb{E},\,\gamma>0.
\]

We let $J\colon{\mathbb{E}}\times \mathbb{R}_{++}\to \dom g$ be the {\em forward-backward operator} for problem~\cref{pr1}, which is given by
\begin{equation}\label{J}
J(x,\gamma):=\prox_{\gamma g}(x-\gamma \nabla f(x)),\quad x\in \mathbb{E},\,\gamma>0.
\end{equation}
It is well known that if $h$ is differentiable and
$\nabla h$ is $L$-Lipschitz continuous on $ \mathbb{E}$, i.e.,
\[
\|\nabla h(x)-\nabla h(y)\|\leq L\|x-y\|, \quad x,y \in \mathbb{E},
\]
then, for all $x, y \in \mathbb{E}$, we have
\begin{equation}\label{eqlq}
h(y)+\langle \nabla h(y), x-y\rangle\leq h(x)\leq h(y)+\langle \nabla h(y), x-y\rangle +\frac{L}{2}\|x-y\|^2.
\end{equation}

The next lemma provides some basic  properties of the subdifferential  operator.

\begin{lemma}\label{for:transp}
Let $h,f \colon \mathbb{E} \to \overline{\mathbb{R}}$ be proper, closed, and convex functions. Then:
\item[ {\bf (i)}]
$\partial_{\ve}h(x)+\partial_{\mu}f(x)\subset\partial_{\ve+\mu}(h+f)(x)$, for all $x\in \mathbb{E}$ and $\ve,\mu\geq0$.

\item [ {\bf (ii)}]
$w\in \partial h(y)$ implies $w\in \partial_\ve h(x)$, where $\ve=h(x)-[h(y)+\inner{w}{x-y}]\ge 0$.

\end{lemma}

The following notion of an approximate solution of problem~\cref{pr1} is used in the complexity analysis of our methods.

\begin{definition}\label{def:approxsol}
Given a tolerance $\rho>0$, a point $ x\in \mathbb{R}^n$ is said to be a $\rho$-approximate solution of problem~\cref{pr1}
with residues $(v,\ve)\in \mathbb{R}^n\times\mathbb{R}_+$ if and only if
\[
v\in \partial_{\ve} F(x), \quad \max\{\|v\|,\ve\}\leq \rho.\label{approxSol1}
\]
\end{definition}

We end this section by presenting some elementary properties on the extrapolate sequences used by the proposed methods.

\begin{lemma}\label{aux-1}
The positive sequence $(t_k)_{k\in \NN}$ generated by \cref{tk-FISTA} satisfies, for all $k\in\mathbb{N}$,
    \item [ {\bf (i)}] $\displaystyle\frac{1}{t_k}\le \frac{2}{k+1}$,
    \item [ {\bf (ii)}] $\displaystyle t^2_{k+1}-t_{k+1}=t^2_k$,
     \item [ {\bf (iii)}] $\displaystyle0\le \frac{t_k-1}{t_{k+1}}\le 1$.
\end{lemma}

\begin{lemma}\label{aux-12} Let $\lambda\ge 1$ be given. The sequence $(\tau_k)_{k\in \NN}$ recursively defined by
\begin{equation}\label{tk-kappa}
\tau_0\coloneqq 0, \quad \mbox{and}\quad \tau_{k+1}\coloneqq \tau_{k}+\frac{ \lambda+\sqrt{\lambda^2+4\lambda \tau_{k}}}{2},
\end{equation}
satisfies, for all $k\in \mathbb{N}$,
    \item [ {\bf (i)}] $\tau_{k+1}>\tau_{k}$ and
 $\displaystyle \frac{\tau_{k+1}}{{(\tau_{k+1}-\tau_{k}})^2}=\frac1\lambda$,
    \item [ {\bf (ii)}] $ \displaystyle \tau_k\geq \frac{\lambda}{4}k^2$.
\end{lemma}
\begin{proof}
The first item follows from definition and the fact that $\tau_k\ge 0$ for all $k\in \mathbb{N}$.
To prove the second item, we first note that
\[
\tau_{k+1} =
\tau_{k}+\frac{\lambda+\sqrt{\lambda^2+4\lambda \tau_{k}}}{2} \geq
\tau_{k}+\frac{ \lambda+2\sqrt{\lambda \tau_{k}}}{2} \geq
\Bigg( \sqrt{\tau_{k}}+\frac{\sqrt{\lambda }}{2} \Bigg)^2,
\]
which implies
\(
\displaystyle\sqrt{\tau_{k+1}}\geq\sqrt{\tau_{k}}+\frac{\sqrt{\lambda }}{2}.
\)
Therefore,
\[
\sqrt{\tau_{k}}\geq\sqrt{\tau_{0}}+\sum_{i=1}^k\frac{\sqrt{\lambda }}{2}=k\frac{\sqrt{\lambda }}{2}.
\]
Squaring both sides, we obtain the second item.
\end{proof}

\section{Inexact criteria with relative error rules}\label{Inexact-Rules}
In this section we present two inexact rules with relative error criteria: the {\em inexact relative rule} (IR Rule) and the {\em inexact extra-step relative rule} (IER Rule). These rules will be used in the two proposed methods in the following two sections.
\medskip
\begin{center}\fbox{\begin{minipage}[b]{0.981\textwidth}
\begin{ex}[IR Rule]
Given $\tau \in  (0,1]$ and $\alpha\in [0,{(1-\tau)L}/{\tau}]$, we define the set-value mapping
$\mathcal{J}^{\alpha,\tau}\colon\mathbb{E}\times \mathbb{R}_{++}\rightrightarrows \mathbb{E}\times \mathbb{E}\times \mathbb{R}_+$ as
\[
\!\mathcal{J}^{\alpha,\tau}\!\!\left(y,\frac1 L\right)\!\!
\coloneqq\!\!
\left\{\!(x,v,\ve)\in \mathbb{E}\times \mathbb{E}\times \mathbb{R}_+ \!\left|
\!\!
\begin{array}{c}
v \in \partial_{\ve}g( x) + \frac{L}\tau(x - y) + \nabla f(y),\\
\\
\|\tau v\|^2 + 2\tau \ve L \leq L[(1 - \tau)L - \alpha\tau]\|x - y\|^2
\end{array}
\!\!\!\!
\right.
\!\right\}\!.
\]
\end{ex}
\end{minipage}}\end{center}
\medskip

Note that the IR Rule consists of (possibly many) specific outputs.
Next we discuss some particular output possibilities including exact and inexact proximal solutions with relative errors.

\begin{remark}\label{Rk1} By setting $v=0$ in the \textup{IR Rule}, we recover the inexact solution of~\cref{inexact-Sum}, but with $L/\tau$ in place of $L$.
In this case, the inclusion is similar to the one in \cite{Villa-Salzo-Luca-Verri-2013}; however, the condition on $\ve$ is different from the exogenous one in \cite{Villa-Salzo-Luca-Verri-2013}.
If $\tau=1$ in the \textup{IR Rule}, then $\alpha=0$, $\ve=0$, and $v=0$, implying that
\[
\mathcal{J}^{0,1}\left(y,1/L\right) =
\Bigg\{
\bigg(
\prox_{\frac{1}{L} g}
\Big(y-\frac{1}{L} \nabla f(y)\Big),
0, 0
\bigg)
\Bigg\}
,
\]
which agrees with the exact {\em prox} used in \eqref{xk-FISTA}.
\end{remark}

\medskip
\begin{center}\fbox{\begin{minipage}[b]{0.981\textwidth}
\begin{ex}[IER Rule]
Given $\sigma \in [0,1]$ and $\alpha>1/L$, we define  the set-value mapping $\mathcal{J}_e^{\alpha,\sigma}\colon\mathbb{E}\times \mathbb{R}_{+}\rightrightarrows \mathbb{E}\times \mathbb{E}\times \mathbb{R}_+$ as
\[
\mathcal{J}_e^{\alpha,\sigma}\left(y,\frac 1 L\right)\coloneqq
\left\{
(\tilde x,v,\ve)\in \mathbb{E}\times \mathbb{E}\times \mathbb{R}_+ \,\left|\,
\begin{array}{c}
v \in \partial_{\ve}g(\tilde x) + {L}(\tilde x - y) + \nabla f(y), \\
\\
\|\alpha v+ \tilde x-y\|^2+2\alpha\ve \leq \sigma^2\|\tilde x-y\|^2\end{array}\right.
\right\}
\!.
\]
\end{ex}
\end{minipage}}\end{center}
\medskip

\begin{remark} \label{Rk2}
By fixing  $v=(y-\tilde x)/\alpha$ in the \textup{IER Rule}, we recover the inexact solution of \cref{inexact-Sum}, but with $(1+\alpha L)/\alpha$ in place of $L$. However, the condition on $\ve$ is different from the exogenous one in \cite{Villa-Salzo-Luca-Verri-2013}.
If $\sigma=0$ in the \textup{IER Rule}, then $\ve=0$ and $v=(y-\tilde x)/\alpha$, where
\[
\tilde x = \prox_{\frac{\alpha}{1+\alpha L} g}
\bigg(
y-\frac{\alpha}{1+\alpha L}\nabla f(y)
\bigg),
\]
implying that
\begin{align*}
\mathcal{J}_e^{\alpha,0}\left(y,1/L\right)
&=
\Bigg\{
\bigg(
\tilde x,
\frac{y - \tilde x}\alpha, 0
\bigg)
\Bigg\}.
\end{align*}
\end{remark}

It is worth pointing out that  the inexact relative rules defined above are nonempty since
 the inclusions  \[0 \in \partial g( x) + \frac{L}\tau(x - y) + \nabla f(y), \quad 0 \in \partial g(\tilde x) +\frac{(1+\alpha L)}{\alpha}(\tilde x - y) + \nabla f(y)\]  always have  solutions, which implies that
\[
\bigg(
\prox_{\frac{\tau}{L} g}
\Big(y-\frac{\tau}{L} \nabla f(y)\Big),
0, 0
\bigg)
\in
\mathcal{J}^{\alpha,\tau}(y, 1/L), \quad \tau \in (0,1], \quad \alpha\in [0,{(1-\tau)L}/{\tau}],
\]
and
\[
\bigg(
\tilde x,
\frac{y - \tilde x}\alpha, 0
\bigg)
\in
\mathcal{J}^{\alpha,\sigma}_e(y, 1/L), \quad \alpha > 1/L, \quad \sigma \in [0,1].
\]

\section{Inexact accelerated method}\label{PFA}

We now formally present our inexact accelerated method.
\medskip
\begin{center}\fbox{\begin{minipage}[b]{0.981\textwidth}
\begin{algorithm}[I-FISTA]\label{alg:1}
Let $x_{0} \in \mathbb{E}$, $\tau\in(0,1]$, and $\alpha\in [0,L(1-\tau)/\tau]$ be given.
Set $y_1:=x_0$, $t_1:=1$, and iterate
\begin{gather}
\text{find}\ (x_k,v_k,\ve_k)\in \mathcal{J}^{\alpha,\tau}(y_k,1/L),\label{deftriple}\\
 t_{k+1}\coloneqq \frac{1+\sqrt{1+4t_k^2}}{2}, \label{deft}\\
 y_{k+1}\coloneqq x_{k}-\left(\frac{ t_k}{t_{k+1}}\right)\frac{\tau}{L}v_{k}+\left(\frac{t_{k}-1}{t_{k+1}}\right)( x_{k}-x_{k-1}).\label{defyy}
\end{gather}
\end{algorithm}
\end{minipage}}\end{center}
\medskip

Note that the triple $(x_k,v_k,\ve_k)$ in the iterative step of {I-FISTA} satisfies
\begin{gather}\label{IR-cond-k-inc}
v_k\in \partial_{\ve_k}g( x_k)+\frac{L}\tau( x_k- y_k)+\nabla f(y_k), \\\label{IR-cond-k-ine}
\|\tau v_k\|^2+2\tau \ve_k L\leq L[(1-\tau) L-\alpha\tau]\| x_k-y_k\|^2.
\end{gather}
If $\tau=1$, then we have $\ve_k = 0$ and $v_k = 0$, giving us
\begin{gather*}
0 \in \partial g(x_k) + L(x_k - y_k) + \nabla f(y_k), \\
y_{k+1} = x_k + \left(\frac{t_{k}-1}{t_{k+1}}\right)(x_{k} - x_{k-1});
\end{gather*}
hence, {I-FISTA} recovers the classical FISTA.

Next we present a key result for our analysis.

\begin{proposition}\label{prop1-2}
For every $x\in \mathbb{E}$ and $ k\in \NN$, we have
\[
F(x)-F( x_{k})\ge \frac{L}{2\tau}\left[\left\| x_{k}-x-\frac{\tau}{L}v_{k}\right\|^2-\|y_{k}-x\|^2\right]+\frac{\alpha}2\|y_{k}- x_{k}\|^2.
\]
\end{proposition}
\begin{proof}
Let $x \in \mathbb{E}$ and $k \in \mathbb{N}$. Note first that from \cref{IR-cond-k-inc}, \[v_{k}+\frac{L}{\tau}{({y}_{k}- x_{k})}-\nabla f({y}_{k})\in\partial_{\ve_{k}} g( x_{k}).\]
From the definition of $\partial_\ve g$, we have
\begin{equation}\label{eq2-2}
g(x)-g( x_{k})\ge \Big\langle v_{k}+\frac{L}{\tau}({y}_{k}- x_{k})-\nabla f({y}_{k}),x- x_{k}\Big\rangle-\ve_{k}.
\end{equation}
Moreover, the convexity of $f$ implies
\begin{equation}\label{eq3-2}
f(x)-f(y_{k})\ge \langle\nabla f(y_{k}), x-y_{k}\rangle.
\end{equation}
Adding \cref{eq2-2} and \cref{eq3-2}, using $F=f+g$, and simplifying, we
get
\begin{align*}
F(x)-F( x_{k})& \geq f(y_{k})- f(x_{k})+ \langle\nabla f(y_{k}), x_{k}-y_{k}\rangle\\
&\quad- \frac{L}\tau\inner{{y}_{k}- x_{k}}{ x_{k}-x}+\inner{v_{k}}{x-x_{k}}-\ve_{k}.
\end{align*}
Combining the above inequality with the following identity
\[-\inner{{y}_{k}- x_{k}}{ x_{k}-x} =\frac{1}{2}\left[\|y_{k}- x_{k}\|^2 +\| x_{k}-x\|^2-\|y_{k}-x\|^2\right],\]
we get that
\begin{align*}
F(x)-F( x_{k})&\geq f(y_{k})- f( x_{k})+
\langle\nabla f(y_{k}), x_{k}-y_{k}\rangle
-\ve_{k}\\
&\quad+ \frac{L}{2\tau}\|y_{k}- x_{k}\|^2
+ \frac{L}{2\tau}\left[\| x_{k}-x\|^2-\|y_{k}-x\|^ 2\right]+\inner{v_{k}}{x- x_{k}}\\
&=f(y_{k})- f( x_{k})+ \langle\nabla f(y_{k}), x_{k}-y_{k}\rangle+ \frac{L}{2}\|y_{k}- x_{k}\|^2\\
&\quad+\frac{(1-\tau)L}{2\tau}\|y_{k}- x_{k}\|^2+ \frac{L}{2\tau}\left[\| x_{k}-x\|^2-\|y_{k}-x\|^ 2\right]\\
&\quad+\inner{v_{k}}{x- x_{k}}-\ve_{k}.
\end{align*}
Then, using \cref{eqlq} together with the Lipschitz continuity of $\nabla f$, we have
\begin{align*}
F(x)-F( x_{k})&\geq \frac{(1-\tau)L}{2\tau}\|y_{k}- x_{k}\|^2+\frac{L}{2\tau}\left[\| x_{k}-x\|^2-\|y_{k}-x\|^ 2\right]\\
&\quad+\inner{v_{k}}{x- x_{k}}-\ve_{k}.
\end{align*}
On the other hand, the error condition of {IR Rule}, given in \cref{IR-cond-k-ine}, implies
\[
 \frac{(1-\tau)L}{2\tau}\|y_{k}- x_{k}\|^2-\ve_{k}\geq \frac{\tau}{2L}\| v_{k}\|^2+\frac{\alpha}2\|y_{k}- x_{k}\|^2.
\]
Hence, combining the last two inequalities, we obtain
\begin{align*}
F(x)-F( x_{k})&\ge \frac{L}{2\tau}\left[\| x_{k}-x\|^2-\|y_{k}-x\|^2\right]+\inner{v_{k}}{x- x_{k}}\\
&\quad+\frac{\tau}{2L}\| v_{k}\|^2+\frac{\alpha}2\|y_{k}- x_{k}\|^2,
\end{align*}
which gives us
\begin{align*}
F(x)-F( x_{k})&\ge\frac{L}{2\tau}\left[\| x_{k}-x\|^2+\frac{2\tau}{L}\inner{v_{k}}{x- x_{k}}+\left\| \frac{\tau}{L} v_{k}\right\|^2-\|y_{k}-x\|^2\right]\\
&\quad+\frac{\alpha}2\|y_{k}- x_{k}\|^2,
\end{align*}
implying that
\[
F(x)-F( x_{k})\ge \frac{L}{2\tau}\left[\left\| x_{k}-x-\frac{\tau}{L}v_{k}\right\|^2-\|y_{k}-x\|^2\right]+\frac{\alpha}2\|y_{k}- x_{k}\|^2,
\]
as desired.
\end{proof}

\begin{theorem}\label{l-rate-2}
Let $(x_k,y_k,t_k)_{k\in \mathbb{N}}$ be the sequence
generated by \textup{I-FISTA}. Then, for all $k\in\mathbb{N}$,
\begin{equation}\label{nn23}
\frac{2\tau}{L}[t_{k}^2( F( x_k)-F^*)-t_{k+1}^2( F( x_{k+1})-F^*)]\ge \|u_{k+1}\|^2- \|u_k\|^2+\frac{\tau \alpha{t_{k+1}^2}}{L}\|y_{k+1}-x_{k+1}\|^2,
\end{equation}
where
\begin{equation}\label{def:uk}
u_k \coloneqq t_{k}( x_{k}- x_{k-1})-\frac{\tau}{L}t_{k}v_{k}+( x_{k-1}-x_*), \quad x_*\in S_*.
\end{equation}
\end{theorem}
\begin{proof}
Let $x_* \in S_*$.
Using \cref{prop1-2} with $k+1$ in place of $k$ and at $x=x_*$ and $x= x_k$, we have
\begin{align*}
-(F( x_{k+1})-F^*) &\ge \frac{L}{2\tau}\left[\left\| x_{k+1}-x_*-\frac{\tau}{L}v_{k+1}\right\|^2-\|y_{k+1}-x_*\|^2\right]\\
&\quad+\frac{\alpha}2\|y_{k+1}- x_{k+1}\|^2,\\
F(x_k)-F( x_{k+1})&\ge \frac{L}{2\tau}\left[\left\| x_{k+1}-x_k-\frac{\tau}{L}v_{k+1}\right\|^2-\|y_{k+1}-x_k\|^2\right]\\
&\quad+\frac{\alpha}2\|y_{k+1}- x_{k+1}\|^2.
\end{align*}
By multiplying the second inequality by $(t_{k+1}-1)$ and adding it to the first inequality above, we obtain
\begin{align*}
(t_{k+1}-1)(&F( x_{k})-F^*)-t_{k+1}(F( x_{k+1})-F^*)\\&\geq \frac{L}{2\tau}\left\| x_{k+1}-x_*-\frac{\tau}{L}v_{k+1}\right\|^2
-\frac{L}{2\tau}\|y_{k+1}-x_*\|^2+ \frac{\alpha t_{k+1}}{2}\|y_{k+1}- x_{k+1}\|^2\\
& \quad + \frac{L(t_{k+1}-1)}{2\tau}\left[\left\| x_{k+1}-x_k-\frac{\tau}{L}v_{k+1}\right\|^2-\|y_{k+1}-x_k\|^2\right].
\end{align*}
Multiplying now by $2\tau t_{k+1}/L$ in the last inequality and then using part (ii) of \cref{aux-1}  (i.e., $t_{k+1}(t_{k+1}-1)=t^2_k$), we have
\begin{align*}
\frac{2\tau}L[t_{k}^2(F( x_{k})-F^*)&-t_{k+1}^2(F( x_{k+1})-F^*)]\ge ({t_{k+1}^2-t_{k+1}})\left\| x_{k+1}-x_k-\frac{\tau}{L}v_{k+1}\right\|^2\\
&\quad- ({t_{k+1}^2-t_{k+1}})\|y_{k+1}- x_k\|^2+ {t_{k+1}}\left\| x_{k+1}-x_*-\frac{\tau}{L}v_{k+1}\right\|^2 \\
&\quad- {t_{k+1}}\|y_{k+1}-x_*\|^2+\frac{\tau \alpha{t_{k+1}^2}}{L}\|y_{k+1}- x_{k+1}\|^2,
\end{align*}
which implies that
\begin{align}\nonumber
\frac{2\tau}L[&t_{k}^2(F( x_{k})-F^*)-t_{k+1}^2(F( x_{k+1})-F^*)]\\\nonumber
&\ge \left\| t_{k+1}(x_{k+1}-x_k)-\frac{\tau}{L}t_{k+1}v_{k+1}\right\|^2-\|t_{k+1}(y_{k+1}- x_k)\|^2\\\nonumber
&\quad+ t_{k+1}\left(\|y_{k+1}- x_k\|^2- \left\|x_{k+1}-x_k-\frac{\tau}{L}v_{k+1}\right\|^2\right)\\\label{eq:5623}
&\quad+ t_{k+1}\left( \left\| x_{k+1}-x_*-\frac{\tau}{L}v_{k+1}\right\|^2-\|y_{k+1}-x_*\|^2\right)
+\frac{\tau \alpha{t_{k+1}^2}}{L}\|y_{k+1}- x_{k+1}\|^2.
\end{align}
Now, from the definitions of $y_{k+1}$ and $u_k$ in \cref{defyy} and \cref{def:uk}, respectively, we have
\begin{align*}
 &\left\| t_{k+1}(x_{k+1}-x_k)-\frac{\tau}{L}t_{k+1}v_{k+1}\right\|^2-\|t_{k+1}(y_{k+1}- x_k)\|^2\\&=\left\|u_{k+1}-(x_k-x_*)\right\|^2
 - \|u_k-(x_k-x_*)\|^2\\
 &=\left\|u_{k+1}\right\|^2- \|u_k\|^2+2\inner{u_{k}-u_{k+1}}{x_k-x_*}\\
  &=\left\|u_{k+1}\right\|^2- \|u_k\|^2+2t_{k+1}\left\langle{y_{k+1}-x_{k+1}+\frac{\tau}{L}v_{k+1}},{x_k-x_*}\right\rangle\\
  &=\left\|u_{k+1}\right\|^2- \|u_k\|^2+2t_{k+1}\left[\left\langle{y_{k+1}-x_{k}},{x_k-x_*}\right\rangle-\left\langle{x_{k+1}-x_{k}-\frac{\tau}{L}v_{k+1}},{x_k-x_*}\right\rangle\right]\\
  &=\left\|u_{k+1}\right\|^2- \|u_k\|^2+ t_{k+1}\left(\|y_{k+1}-x_*\|^2-\|y_{k+1}- x_k\|^2\right)\\
  &\quad+ t_{k+1}\left(\left\|x_{k+1}-x_k-\frac{\tau}{L}v_{k+1}\right\|^2- \left\| x_{k+1}-x_*-\frac{\tau}{L}v_{k+1}\right\|^2\right).
\end{align*}
Therefore, \cref{nn23} now follows from \eqref{eq:5623} and the last equality.
\end{proof}

\begin{theorem}
Let $d_0$ be the distance from $x_0$ to $S_*$.
Let $(x_k,y_k,t_k)_{k\in \mathbb{N}}$ be the sequence
generated by \textup{I-FISTA}. Then,  for all $k\in\mathbb{N}$,
\begin{equation}\label{nn2354}
t_k^2(F(x_{k})-F^*)+\frac\alpha2\sum_{i=1}^{k}t_{i}^2\|y_{i}-x_{i}\|^2\leq \frac{L}{2\tau}d_0^2.
\end{equation}
In particular,
\begin{equation}\label{nn2354**}
F(x_{k})-F^*\leq \frac{2L}{\tau (k+1)^2}d_0^2.
\end{equation}
\end{theorem}
\begin{proof}
Summing \cref{nn23} in \cref{l-rate-2} from $k:=1$ to $k:=k-1$, and using the fact that $t_1 = 1$, we obtain
\begin{equation}\label{eq:67454}
\frac{2\tau}L t_{k}^2(F(x_{k})-F^*)+\|u_k\|^2+\frac{\tau \alpha}{L} \sum_{i=2}^{k}t_{i}^2\|y_{i}-x_{i}\|^2\leq\frac{2\tau}L(F(x_{1})-F^*)+ \|u_{1}\|^2.
\end{equation}
Now let $x_*$ be the projection of $x_0$ onto $S_*$. Then $d_0 = \|x_0 - x_*\|$.
From \cref{prop1-2} at $k=1$ and $x=x_*$, and using the fact that $y_1=x_0$,   $u_1=x_1-x_*-\frac{\tau}{L}v_{1}$, and $t_1 = 1$, we have that
\begin{align*}
\frac{2\tau}L (F(x_{1})-F^*)&\leq \|y_{1}-x_*\|^2-\left\|x_{1}-x_*-\frac{\tau}{L}v_{1}\right\|^2-\frac{\tau \alpha}{L}\|y_{1}-x_{1}\|^2\\
&= \|x_0-x_*\|^2-\|u_{1}\|^2-\frac{\tau \alpha}{L}t_1^2\|y_{1}-x_{1}\|^2.
\end{align*}
This inequality together with \cref{eq:67454} imply \cref{nn2354}. To prove \cref{nn2354**}, note that part (i) of \cref{aux-1} implies $t_k\ge \frac{k+1}{2}$, hence the result follows directly from \cref{nn2354}.
\end{proof}

We next derive iteration-complexity bounds for I-FISTA to obtain approximate solutions of problem~\cref{pr1} in the sense of \cref{def:approxsol}.

\begin{theorem}\label{l-rate-456}
Let $d_0$ be the distance from $x_0$ to $S_*$.
Let $(x_k,y_k,t_k)_{k\in \mathbb{N}}$ be the sequence
generated by \textup{I-FISTA}.  Then, for every $k\in \mathbb{N}$,
\[
r_k\in \partial_{\ve_k}g(x_{k})+\nabla f(x_k) \subset \partial_{\ve_k}F(x_{k}),\]
where $r_k\coloneqq v_k+{L}(y_{k}- {x}_k)/\tau+\nabla f( x_k)-\nabla f( y_k)$. Additionally, if  $\tau<1$ and $\alpha\in (0,L(1-\tau)/\tau]$, then
 there exists $\ell_k\leq k$ such that
\begin{equation}\label{nn2354er3452}
\|r_{\ell_k}\|=\mathcal{O}\Big(d_0\sqrt{L^3/k^3}\Big), \quad \ve_{\ell_k}=\mathcal{O}\big(d_0^2{L^2/k^3}\big),
\end{equation}
\end{theorem}
\begin{proof}
The inclusion follows from \cref{IR-cond-k-inc}.
Now let $x_*$ be the projection of $x_0$ onto $S_*$.
It follows from  \cref{nn2354} that
\[
\min_{i=1, \ldots,k}\|y_{i}-x_{i}\|^2
\leq
\frac{L}{\alpha\tau\sum_{i=1}^{k}t_{i}^2}d_0^2,
\]
which, when combined with part (i) of \cref{aux-1}, yields
\[
\min_{i=1, \ldots,k}\|y_{i}-x_{i}\|^2
\leq \frac{4L}{\alpha\tau\sum_{i=1}^{k}(i+1)^2}d_0^2.
\]
Since
\[
\sum_{i=1}^{k}(i+1)^2 = \frac{k(k+1)(2k+1)}{6}+k(k+2)  \ge \frac{k^3}{3}, \quad \forall k \ge 1,
\]
we obtain
\[
\min_{i=1, \ldots,k}\|y_{i}-x_{i}\|^2
\le \frac{12L}{\alpha\tau k^3}d_0^2.
\]
Hence, there exists ${\ell_k}\leq k$ such that
\begin{equation}\label{yk-xk}
\|y_{{\ell_k}}-x_{{\ell_k}}\|\leq 2\sqrt{\frac{3L}{\alpha\tau k^3}}\,d_0.
\end{equation}
From the definition of $r_k$, condition \cref{IR-cond-k-ine} for $\|v_k\|$ in the {IR Rule}, and the Lipschitz continuity of $\nabla f$, we have
\begin{align*}
\|r_{\ell_k}\|&\leq \|v_{\ell_k}\|+\frac{L}{\tau}\|y_{{\ell_k}}- {x}_{\ell_k}\|+\|\nabla f( x_{\ell_k})-\nabla f( y_{\ell_k})\|\\
& \leq \left( \frac{\sqrt{L[(1-\tau) L-\alpha\tau]} }{\tau}+\frac{L}{\tau}+L\right)\|y_{{\ell_k}}-x_{{\ell_k}}\|\\&
\leq 2L\left( \frac{\sqrt{1-\tau} + 1 + \tau}{\tau}\right)\sqrt{\frac{3L}{\alpha\tau k^3}}\,d_0,
\end{align*}
which implies the first part of \cref{nn2354er3452}. Moreover, it follows from condition \cref{IR-cond-k-ine} for $\ve_k$ in the IR Rule that
\[
\ve_{\ell_k}\leq \frac{ { (1-\tau)}L-\alpha\tau}{2\tau}\|x_{\ell_k}-y_{\ell_k}\|^2\leq \frac{6L\left[ { (1-\tau)}L-\alpha\tau\right]}{\alpha\tau^2 k^3}d_0^2,
\]
which proves the second part of \cref{nn2354er3452}.
\end{proof}

\section{Inexact extragradient accelerated method}\label{IEAM}

We now formally present our inexact accelerated me\-thod with an extra-step.
\medskip
\begin{center}\fbox{\begin{minipage}[b]{0.981\textwidth}
\begin{algorithm}[IE-FISTA] \label{alg:2}
Let $x_0$, $y_0 \in \mathbb{E}$, $\alpha>1/L$ and $\sigma\in[0,1]$ be given, and set $\lambda\coloneqq \alpha/(1+\alpha L)$, $\tau_0\coloneqq 0$, $\tilde x_0 \coloneqq x_0$ and $k\coloneqq 0$. \\
\noindent Iterative Step. Compute
\begin{align}\label{deft1}
    \tau_{k+1}&\coloneqq \tau_{k}+\frac{ \lambda+\sqrt{\lambda^2+4\lambda \tau_{k}}}{2},\\
    y_{k}&\coloneqq \frac{ \tau_{k}}{ \tau_{k+1}}\tilde x_{k}+\frac{ \tau_{k+1}- \tau_{k}}{\tau_{k+1}}x_{k}, \label{defy2}
\end{align}
and find a triple
 $$(\tilde x_{k+1},v_{k+1},\ve_{k+1})\in \mathcal{J}_e^{\alpha,\sigma}(y_k,1/L)$$ given in {IER Rule},
and set
\begin{equation} \label{defx2}
 x_{k+1}\coloneqq x_{k}-(\tau_{k+1}-\tau_k)(v_{k+1}+L( {y}_k-\tilde x_{k+1})).
\end{equation}
\end{algorithm}
\end{minipage}}\end{center}
\medskip

Note that the triple $(\tilde x_{k+1},v_{k+1},\ve_{k+1})$ in the iterative step of \textup{IE-FISTA} satisfies
\begin{gather}\label{defx}
v_{k+1}\in \partial_{\ve_{k+1}}g(\tilde x_{k+1})+{L}(\tilde x_{k+1}- {y}_k)+\nabla f( y_k), \\\label{deferro}
\|\alpha v_{k+1}+ \tilde x_{k+1}-y_k\|^2+2\alpha\ve_{k+1} \leq \sigma^2\|\tilde x_{k+1}-y_k\|^2.
\end{gather}
If $\sigma=0$, it follows from \cref{deferro} that $\ve_{k+1}=0$ and $ v_{k+1}=(y_k- \tilde x_{k+1})/\alpha$, giving us
\begin{gather*}
\tilde{x}_{k+1}=\argmin_{x \in \mathbb{E}} \left \{g(x)+\frac{1}{2\lambda}\left\|x-\left( {y}_k-\lambda \nabla f( y_k)\right)\right\|^2 \right \}, \\
x_{k+1} = x_{k}-\frac{(\tau_{k+1}-\tau_k)}{\lambda}( {y}_k-\tilde x_{k+1});
\end{gather*}
hence, \textup{IE-FISTA} recovers the exact version proposed in \cite[Algorithm~I]{acc-HPE}.

We begin the complexity analysis of {IE-FISTA} by first defining the sequence $(\mu_k)_{k\in\mathbb{N}}$ as
\begin{equation}\label{def:mu}
\mu_{k}\coloneqq f(\tilde x_{k})-\big[f(y_{k-1})+\inner{\nabla f(y_{k-1})}{\tilde x_{k}-y_{k-1}}\big], \quad \forall \, k\in\mathbb{N}.
\end{equation}
We also consider the affine maps $\Psi_k\colon\mathbb{E}\to \mathbb{R}$ given by
\begin{equation}\label{defgamma}
\Psi_k(x)\coloneqq F(\tilde x_k)+\inner{v_{k}+L( {y}_{k-1}-\tilde x_{k})}{x-\tilde x_k}-\mu_k-\ve_k, \quad \forall \, x\in \mathbb{E}\;\mbox{and} \; k\in\mathbb{N},
\end{equation}
and $\Gamma_k\colon\mathbb{E}\to \mathbb{R}$ defined as
\begin{equation}\label{defGamma}
 \Gamma_0(x)\equiv 0, \quad \Gamma_{k+1}(x):=\frac{\tau_{k}}{\tau_{k+1}} \Gamma_{k}(x)+\frac{\tau_{k+1}- \tau_{k}}{ \tau_{k+1}}\Psi_{k+1}(x), \quad \forall \, x\in \mathbb{E}\;\mbox{and} \; k\geq0.
\end{equation}

\begin{lemma}\label{lem-aux} Let $(x_k,\tilde x_k, y_k)_{k\in\mathbb{N}}$ be the sequence generated by \textup{IE-FISTA}. Then the following hold.
\item [ {\bf (i)}] For all $k\in\mathbb{N}$,
\begin{equation}\label{eq:rt65}
\mu_{k}\leq \frac L2 \| \tilde x_{k}-y_{k-1}\|^2.
\end{equation}
\item [ {\bf (ii)}] For all $k\geq0$, \begin{equation}\label{def:xprob}
x_k= \argmin_{x\in\mathbb{E} }\left\{ \tau_k\Gamma_k(x)+\frac12\|x-x_0\|^2\right\}.
\end{equation}
\end{lemma}
\begin{proof}
For part (i), we have that inequality \cref{eq:rt65} follows from \cref{def:mu} and \cref{eqlq}.
To prove part (ii), we first observe that \cref{defgamma} and \cref{defGamma} imply that
\begin{equation}\label{pro:gamma}
\tau_k\nabla\Gamma_k(x)=\sum_{i=1}^k(\tau_{i}- \tau_{i-1})(v_{i}+L( {y}_{i-1}-\tilde x_{i})),
\quad \forall \, x\in \mathbb{E}\;\mbox{and} \; k\in \mathbb{N}.
\end{equation}
Combining \cref{pro:gamma} with \cref{defx2} implies
$$
x_k=x_0-\sum_{i=1}^k(\tau_{i}- \tau_{i-1})(v_{i}+L( {y}_{i-1}-\tilde x_{i}))=x_0-\tau_k\nabla \Gamma_k(x).
$$
Hence,
$0=\tau_k\nabla \Gamma_k(x)+x_k-x_0$, which proves \cref{def:xprob}.
\end{proof}

\begin{lemma}\label{pro:gammaL}
Let $(x_k,\tilde x_k, y_k)_{k\in\mathbb{N}}$ be the sequence generated by \textup{IE-FISTA}. Then the following hold.
\item [ {\bf (i)}] For all $k\in\mathbb{N}$,
\begin{equation}\label{pro:gamma22}
\Psi_k(x)\leq F(x), \quad \forall x\in \mathbb{E}.
\end{equation}
\item [ {\bf (ii)}] For all $k\geq0$,
\begin{equation}\label{pro:gamma23}
\tau_k\Gamma_k(x)\leq \tau_kF(x), \quad \forall x\in \mathbb{E}.
\end{equation}
\end{lemma}
\begin{proof}
First note that from part (ii) of \cref{for:transp} and from the definition of $\mu_k$ \cref{def:mu}, we have that $\nabla f( y_{k-1})\in\partial_{\mu_k}f(\tilde x_{k})$. Hence, it follows from \cref{defx} and part (i) of \cref{for:transp} that
\begin{equation}\label{eqr:45609}
v_{k}+L( {y}_{k-1}-\tilde x_{k}) \in \partial_{\ve_{k}} g(\tilde x_{k})+\nabla f( y_{k-1})\subset \partial_{\ve_k+\mu_k}F(\tilde x_{k}),
\end{equation}
which is equivalent to
$$
F(\tilde x_k)+\langle v_{k}+L( {y}_{k-1}-\tilde x_{k}), x-\tilde x_k\rangle -\mu_k-\epsilon_k\le F(x), \quad \forall \, x\in \mathbb{E}.
$$
Thus \cref{pro:gamma22} follows from the definition of $\Psi_k(x)$ given in \cref{defgamma}, which proves part (i).

To prove part (ii), we use \cref{defGamma} and write
\begin{align*}\tau_k\Gamma_k(x)&=\tau_{k-1}\Gamma_{k-1}(x)+(\tau_{k}- \tau_{k-1})\Psi_{k}(x)\\
&=\sum_{i=1}^k(\tau_{i}- \tau_{i-1})\Psi_{i}(x).
\end{align*}
Then, using item (i), we obtain \cref{pro:gamma23}, which concludes the proof.
\end{proof}

We next establish a key result for the complexity analysis of {IE-FISTA}.

\begin{proposition}\label{lemin453} For every $k\geq0$, let
\begin{equation}\label{def:beta}
\beta_k: =\min_{x\in\mathbb{E} }\left\{ \tau_k\Gamma_k(x)+\frac12\|x-x_0\|^2\right\}-\tau_{k}F(\tilde x_{k}).
\end{equation}
Then,
\begin{equation}\label{ine4356}
\beta_{k+1} \geq \beta_k+ \frac{(1-\sigma^2)\tau_{k+1}}{2\alpha}\|\tilde x_{k+1}-y_k\|^2.
\end{equation}
\end{proposition}
\begin{proof} Let $ u \in \mathbb{E}$. Using the definition of $\Gamma_k$ in \cref{defGamma}, we obtain
\begin{align}\nonumber
\tau_{k+1}\Gamma_{k+1}(u)&+\frac{1}{2}\|u-x_0\|^2= \tau_k\Gamma_{k}(u)+\frac{1}{2}\|u-x_0\|^2+({\tau_{k+1}- \tau_{k}})\Psi_{k+1}(u) \\
&= \tau_k\Gamma_{k}(x_k)+\frac{1}{2}\|x_k-x_0\|^2+\frac{1}{2}\|u-x_k\|^2+({\tau_{k+1}- \tau_{k}})\Psi_{k+1}(u),\label{eqer:er}
\end{align}
where the last equality is due to the fact that $x_k$ is the minimum point of the quadratic function $ \tau_k\Gamma_k(x)+\|x-x_0\|^2/2$ (see part (ii) of \cref{lem-aux}). Next, using part (i) of \cref{pro:gammaL} and the fact that $\Psi_k$ is an affine function, we have
\begin{align*}
({\tau_{k+1}- \tau_{k}})\Psi_{k+1}(u) &\geq({\tau_{k+1}- \tau_{k}})\Psi_{k+1}(u)+\tau_{k}\Psi_{k+1}(\tilde x_{k})- \tau_kF(\tilde x_{k})\\
&= \tau_{k+1}\Psi_{k+1}\left(\frac{\tau_{k}}{\tau_{k+1}}\tilde x_{k}+\frac{\tau_{k+1}- \tau_{k}}{\tau_{k+1}}u\right)-\tau_kF(\tilde x_{k}).
\end{align*}
Now we define
\[
 \tilde c(u)\coloneqq \frac{\tau_{k}}{\tau_{k+1}}\tilde x_{k}+\frac{\tau_{k+1}- \tau_{k}}{\tau_{k+1}}u
\]
and use the definition of $y_k$ in \cref{defy2} to obtain
\begin{align*}
&({\tau_{k+1}- \tau_{k}})\Psi_{k+1}(u)+\frac{1}{2}\|u-x_k\|^2 \\
&\qquad\qquad\geq \tau_{k+1}\bigg(\Psi_{k+1}(\tilde c(u))+\frac{ \tau_{k+1}}{2(\tau_{k+1}- \tau_{k})^2}\|\tilde c(u)- y_{k}\|^2\bigg)
-\tau_kF(\tilde x_{k}).
\end{align*}
Hence, it follows from \cref{eqer:er} and item (i) from \cref{aux-12} that
\begin{align*}
\tau_{k+1}\Gamma_{k+1}(u)+\frac{1}{2}\|u-x_0\|^2
&\geq \tau_k\Gamma_{k}(x_k)+\frac{1}{2}\|x_k-x_0\|^2-\tau_kF(\tilde x_{k})\\
&\quad + \tau_{k+1}
\bigg(
\Psi_{k+1}\,(\tilde c(u))+\frac{1}{2\lambda}\|\tilde c(u)- y_{k}\|^2
\bigg).
\end{align*}
Now, using \cref{def:xprob} and the definitions of $\beta_k$ and $\Psi_k$ in \cref{def:beta} and \cref{defgamma}, respectively, we have
\begin{align*}
&\tau_{k+1}\Gamma_{k+1}(u) + \frac{1}{2}\|u-x_0\|^2 - \tau_{k+1}F(\tilde x_{k+1}) \geq \\
&\quad
\beta_k + \tau_{k+1}
\bigg(
\inner{v_{k+1}+L( {y}_k-\tilde x_{k+1})}{\tilde c(u)-\tilde x_{k+1}}-\mu_{k+1}-\ve_{k+1}
+\frac{1}{2\lambda}\|\tilde c(u)- y_{k}\|^2
\bigg).
\end{align*}
From part (i) of \cref{lem-aux}, we find that
\begin{align*}
L \inner{ {y}_k-\tilde x_{k+1}}{\tilde c(u)-\tilde x_{k+1}}-\mu_{k+1}
&\geq
-L \inner{ {y}_k-\tilde x_{k+1}}{\tilde x_{k+1} - \tilde c(u)} - \frac L2 \| \tilde x_{k+1}-y_k\|^2\\
&= -\frac{L}{2}\|y_k - \tilde c(u)\|^2 +
    \frac{L}{2}\|\tilde x_{k+1} - \tilde c(u)\|\\
&\geq -\frac L{2} \|y_{k}-\tilde c(u)\|^2.
\end{align*}
Combining the last two inequalities, we obtain

\begin{align}\nonumber
&\tau_{k+1}\Gamma_{k+1}(u) + \frac{1}{2}\|u-x_0\|^2 - \tau_{k+1}F(\tilde x_{k+1}) \\
&\qquad
\geq
\beta_k + \tau_{k+1}
\bigg(
\inner{v_{k+1}}{\tilde c(u)-\tilde x_{k+1}}-\ve_{k+1}
+\frac12
\bigg(
\frac{1}{\lambda}-L
\bigg)
\|\tilde c(u)- y_{k}\|^2
\bigg).\label{dq:345}
\end{align}
Now, it follows  from  \cref{deferro}  that
\begin{align*}
&\frac{1-\sigma^2}{2\alpha}\| \tilde x_{k+1}-y_k\|^2\leq-\ve_{k+1}+ \inner{ v_{k+1}}{ y_k-\tilde x_{k+1}}-\frac\alpha2\|v_{k+1}\|^2\\
&=-\ve_{k+1}+ \inner{ v_{k+1}}{  \tilde c(u)-\tilde x_{k+1}}-\frac1{2\alpha}\|\alpha v_{k+1}-( y_k- \tilde c(u))\|^2+\frac1{2\alpha}\| y_k- \tilde c(u)\|^2,
\end{align*}
which implies that
\[
\inner{ v_{k+1}}{ \tilde c(u)-\tilde x_{k+1}}-\ve_{k+1} \geq \frac{1-\sigma^2}{2\alpha}\| \tilde x_{k+1}-y_k\|^2-\frac{1}{2\alpha}\| \tilde c(u)- y_k\|^2.
\]
Combining  \eqref{dq:345} and  the last  inequality, we have
\begin{align*}
&\tau_{k+1}\Gamma_{k+1}(u) + \frac{1}{2}\|u-x_0\|^2 - \tau_{k+1}F(\tilde x_{k+1}) \\
&\qquad
\geq
\beta_k+
\frac{\tau_{k+1}}2
\bigg(
\frac{1-\sigma^2}{\alpha}\| \tilde x_{k+1}-y_k\|^2+ \bigg(
\frac{1}{\lambda} - L - \frac1\alpha
\bigg)
\|\tilde c(u)- y_{k}\|^2
\bigg).
\end{align*}
Now, using the fact that $\lambda= \alpha/(1+\alpha L)$, we obtain
\[
\tau_{k+1}\Gamma_{k+1}(u) + \frac{1}{2}\|u-x_0\|^2 - \tau_{k+1}F(\tilde x_{k+1})
\geq
\beta_k+
\frac{\tau_{k+1}}2
\bigg(
\frac{1-\sigma^2}{\alpha}\| \tilde x_{k+1}-y_k\|^2
\bigg).
\]
Since $u \in \mathbb{E}$ was chosen arbitrarily, this inequality holds for all $u$. Thus, using \cref{def:beta}, we conclude that
\[
\beta_{k+1} \geq \beta_k+ \frac{(1-\sigma^2)\tau_{k+1}}{2\alpha}\|\tilde x_{k+1}-y_k\|^2,
\]
which is the desired inequality.
\end{proof}

The next result establishes the optimal convergence rate of $F(\tilde x_k)-F^*$.

\begin{theorem}\label{l-rate-232}
Let $d_0$ be the distance from $x_0$ to $S_*$.
Let $(x_k,\tau_k,y_k)_{k\in \mathbb{N}}$ be the sequence
generated by \textup{IE-FISTA}. Then,
\begin{equation}\label{nn23542}
\frac{1}{2}\|x_{k}-x_*\|^2+\tau_k(F(\tilde x_{k})-F^*)+ \frac{1-\sigma^2}{2\alpha}\sum_{i=1}^{k}\tau_{i}\|\tilde x_{i}-y_{i-1}\|^2\leq \frac{1}{2}d_0^2.
\end{equation}
In particular,
\[
F(\tilde x_{k})-F^*\leq \frac{2(1+\alpha L)}{\alpha k^2}d_0^2.
\]
\end{theorem}
\begin{proof}
Let $x_*$ be the projection of $x_0$ onto $S_*$.
Using \cref{ine4356} recursively and the fact that $\beta_0=0$, we have
\begin{equation}\label{eq:67454*}
\beta_k\geq \frac{1-\sigma^2}{2\alpha}\sum_{i=1}^{k}\tau_{i}\|\tilde x_{i}-y_{i-1}\|^2.
\end{equation}
From \cref{def:xprob} we have that $x_k$ is the minimum point of the quadratic function $ \tau_k\Gamma_k(x)+\|x-x_0\|^2/2$ and
\[
 \tau_k\Gamma_{k}(x_*)+\frac{1}{2}\|x_*-x_0\|^2 =\min_{x\in\mathbb{E} }\left\{ \tau_k\Gamma_k(x)+\frac12\|x-x_0\|^2\right\}+\frac{1}{2}\|x_*-x_k\|^2.
\]
Combining this with \cref{eq:67454*} and \cref{def:beta} yields
\[
\frac{1}{2}\|x_{k}-x_*\|^2+\tau_k(F(\tilde x_{k})-\Gamma_{k}(x_*))+ \frac{1-\sigma^2}{2\alpha}\sum_{i=1}^{k}\tau_{i}\|\tilde x_{i}-y_{i-1}\|^2\leq \frac{1}{2}d_0^2.
\]
Hence, inequality~\cref{nn23542} follows from part (ii) of \cref{pro:gammaL}.

The second part of the theorem follows from the first part, and by part (ii) of \cref{aux-12} and the fact that $\lambda\coloneqq \alpha/(1+\alpha L)$.
\end{proof}

We now present iteration-complexity bounds for {IE-FISTA} to obtain approximate solutions of \cref{pr1} in the sense of \cref{def:approxsol}.

\begin{theorem}
Let $(x_k,\tau_k,y_k)_{k\in \NN}$ be the sequence
generated by \textup{IE-FISTA}.  Then,
\begin{equation} \label{eq:45}
r_k\in \partial_{\ve_{k}} g(\tilde x_{k})+\nabla f( y_{k-1})\subset \partial_{\ve_k+\mu_k}F(\tilde x_{k}), \quad k\in \mathbb{N},
\end{equation}
where $r_k\coloneqq v_{k}+L( {y}_{k-1}-\tilde x_{k})$. Additionally, if $\sigma<1$, then \textup{IE-FISTA} generates a $\rho$-approximate solution $\tilde x_\ell$ of problem~\cref{pr1} with residues $(r_\ell,\ve_\ell+\mu_\ell)$ in the sense of \cref{def:approxsol}
in at most $k=\mathcal{O}\left((d_0/\rho)^{2/3}\right)$ iterations, where $\rho\in(0,1)$ is a given tolerance and
$d_0$ is the distance from $x_0$ to $S_*$.
\end{theorem}
\begin{proof}
The first statement of the theorem follows from \cref{eqr:45609} and the definition of $r_k$.
It follows from  \cref{nn23542} that
\[
\min_{i=1, \ldots,k}\|\tilde x_{i}-y_{i-1}\|^2\leq \frac{\alpha}{(1-\sigma^2)\sum_{i=1}^{k}\tau_{i}}d_0^2,\]
which, when combined with part (ii) of \cref{aux-12}, yields
\[
\min_{i=1, \ldots,k}\|\tilde x_{i}-y_{i-1}\|^2\leq \frac{4\alpha}{\lambda(1-\sigma^2)\sum_{i=1}^{k}i^2}d_0^2.
\]
Since
\[
\sum_{i=1}^{k}i^2 = \frac{k(k+1)(2k+1)}{6} \ge \frac{k^3}{3}, \quad \forall k \ge 1,
\]
we obtain
\[
\min_{i=1, \ldots,k}\|\tilde x_{i}-y_{i-1}\|^2\leq \frac{12\alpha}{\lambda(1-\sigma^2)k^3}d_0^2.
\]
Hence, there exists $1 \leq \ell \leq k$ such that
\begin{equation}\label{eq:g2523}
\|\tilde x_{\ell}-y_{\ell-1}\|^2\leq
\frac{12\alpha}{\lambda(1-\sigma^2)k^3}d_0^2.
\end{equation}
Since the error condition in \cref{deferro} implies that
\[
\|\alpha v_{\ell}\|- \|\tilde x_{\ell}-y_{\ell-1}\| \leq\|\alpha v_{\ell}+ \tilde x_{\ell}-y_{\ell-1}\| \leq \sigma\|\tilde x_{\ell}-y_{\ell-1}\|,
\]
we obtain, from the definition of $r_k$, that
\[
\|r_\ell\|\leq \|v_\ell\|+{L}\|y_{\ell-1}- \tilde x_\ell\| \leq \left( \frac{1+\sigma}{\alpha}+L\right)\|\tilde x_{\ell}-y_{\ell-1}\|.
\]
It then follows from \cref{eq:g2523} that
\[
\|r_\ell\|
\leq \left( \frac{1+\sigma}{\alpha}+L\right)\sqrt{ \frac{12\alpha}{\lambda(1-\sigma^2)}}\frac{d_0}{k^{3/2}}.
\]
In addition, from \cref{deferro}, \cref{eq:g2523}, and $\lambda = \alpha/(1 + \alpha L)$, we have that
\[
\ve_\ell
\leq \frac{\sigma^2}{2\alpha}\|\tilde x_\ell-y_{\ell-1}\|^2
\leq \frac{6\sigma^2}{\lambda(1-\sigma^2)k^3}d_0^2
= \frac{6(1+\alpha L)\sigma^2}{\alpha(1-\sigma^2)k^3}d_0^2.
\]
Moreover, \cref{eq:rt65}, \cref{eq:g2523}, and  $\lambda = \alpha/(1 + \alpha L)$ gives us that
\[
\mu_{\ell}
\leq \frac L2 \| \tilde x_{\ell}-y_{\ell-1}\|^2
\leq \frac{6\alpha L}{\lambda(1-\sigma^2)k^3}d_0^2
= \frac{6L(1+\alpha L)}{(1-\sigma^2)k^3}d_0^2.
\]
Combining the last two inequalities, we have
\[
\ve_\ell + \mu_\ell
\leq
\frac{6(1 + \alpha L)(\sigma^2 + \alpha L)}{\alpha(1 - \sigma^2)k^3} d_0^2.
\]
Choosing $k$ so that
\[
\max\Bigg\{
\left( \frac{1+\sigma}{\alpha}+L\right)\sqrt{ \frac{12\alpha}{\lambda(1-\sigma^2)}}\frac{d_0}{k^{3/2}},
\frac{6(1 + \alpha L)(\sigma^2 + \alpha L)}{\alpha(1 - \sigma^2)k^3} d_0^2
\Bigg\}
\leq \rho,
\]
gives us
\[
r_\ell \in \partial_{\ve_\ell + \mu_\ell} F(\tilde x_\ell),
\quad
\max\{\|r_\ell\|, \ve_\ell + \mu_\ell\} \leq \rho,
\]
which implies that $\tilde x_\ell$ is a $\rho$-approximate solution of problem~\cref{pr1} with residues $(r_\ell, \ve_\ell + \mu_\ell)$.
\end{proof}

\section{Numerical experiments}\label{NumSec}

In this section we explore the numerical behavior of \cref{alg:1} (I-FISTA) and \cref{alg:2} (IE-FISTA) and compare them to the inexact method  with $H_k=L\Id$ described in \cite{Defeng-S-2012} that uses the \emph{inexact absolute rule \textup{(IA Rule)}},
\[
v_k \in \partial_{\ve_k}g(x_k) + L(x
_k - y_k) + \nabla f(y_k),
\qquad
\frac{1}{\sqrt{L}}\|v_k\| \le \frac{\delta_k}{\sqrt{2} t_k},
\qquad
\ve_k = \frac{\xi_k}{2 t_k^2},
\]
where $(\delta_k)_{k\in \mathbb{N}}$ and $(\xi_k)_{k\in \mathbb{N}}$ are summable sequences of nonnegative numbers. In our numerical tests, we use $\delta_k = t_k^{-2}$; by part (i) of \cref{aux-1}, this choice for $(\delta_k)_{k\in \mathbb{N}}$ is summable.
We explain in detail below how $\ve_k$ is computed. Based on this choice for $\ve_k$, we can expect $\ve_k$ to be quite small; in numerical tests we observed $\ve_k$ to be approximately machine epsilon. For this reason, we do not explicitly enforce the above condition on $\ve_k$ in our implementation.
We will refer to the algorithm using the IA Rule as IA-FISTA.

We follow \cite{Defeng-S-2012} by considering the $H$-weighted nearest correlation matrix problem for our numerical tests. All algorithms were implemented in the Julia language \cite{Bezanson2017a} and all tests were run on a machine with a 2.9 GHz Dual-Core Intel Core i5 processor and 16 GB 1867 MHz DDR3 memory.

It is important to note that the goal of this section is not to demonstrate that the code we developed is state-of-the-art for solving the $H$-weighted nearest correlation matrix problem. Rather our goal is to investigate how three different theoretical algorithms perform in practice, giving us insight beyond the convergence results presented in this paper. This is especially interesting since these three algorithms all have the same optimal rate of convergence. Here we see if they can be distinguished by their numerical performance on a set of test instances of the $H$-weighted nearest correlation matrix problem.

\subsection{The nearest correlation matrix problem}

Let $\mathcal{S}^n$ be the set of $n \times n$ real symmetric matrices. Let $G, H \in \mathcal{S}^n$ and define $f \colon \mathcal{S}^n \to \mathbb{R}$ by
\[
f(X) = \frac{1}{2}\| H \circ (X - G) \|_F^2,
\]
where $\circ$ is the Hadamard product and $\|\cdot\|_F$ is the Frobenius norm. We seek the minimizer of $f$ over the set $C$ of $n \times n$ correlation matrices, which is defined as the set of $n \times n$ symmetric positive semidefinite matrices with all ones on the diagonal; that is,
\[
C := \{ X \in \mathcal{S}^n \mid \diag(X) = e, X \succeq 0 \},
\]
where $e \in \mathbb{R}^n$ is the vector of all ones and $\diag \colon \mathcal{S}^n \to \mathbb{R}^n$ is the linear map that returns the vector along the diagonal of the input matrix. The adjoint linear map of $\diag$ is $\Diag \colon \mathbb{R}^n \to \mathcal{S}^n$ which maps a vector of length $n$ to the $n \times n$ diagonal matrix having that vector along its diagonal; indeed, it is easy to verify that $\langle v, \diag(M) \rangle = \langle \Diag(v), M \rangle$ for all $v \in \mathbb{R}^n$ and $M \in \mathcal{S}^n$, where the vector inner-product is $\langle x, y \rangle := x^T y$ for $x, y \in \mathbb{R}^n$ and the symmetric matrix inner-product is $\langle X, Y \rangle := \trace(XY)$ for $X, Y \in \mathcal{S}^n$.
Let $g \colon \mathcal{S}^n \to \mathbb{R} \cup \{+\infty\}$ be defined by
\[
g(X) = \delta_C(X) =
\begin{cases}
0, & X \in C, \\
+\infty, & X \not\in C.
\end{cases}
\]
The $H$-weighted nearest correlation matrix (H-NCM) problem is
\begin{equation}\label{HNCM}
\min_{X \in C} f(X) = \min_{X \in \mathcal{S}^n} f(X) + g(X).
\end{equation}
Note that the gradient of $f$ is given by $\nabla f(X) = H \circ H \circ (X - G)$ and has Lipschitz constant $L := \|H \circ H\|_F$. The KKT optimality conditions for \cref{HNCM} are given by
\begin{gather*}
\nabla f(X) - \Diag(y) - \Lambda = 0,
\\
\diag(X) = e, \quad X \succeq 0, \quad \Lambda \succeq 0, \quad \langle  \Lambda, X \rangle = 0.
\end{gather*}

\subsection{The subproblem}

The subproblem at $Y \in \mathcal{S}^n$ is given by
\[
\min_{X \in \mathcal{S}^n} f(Y) + \langle \nabla f(Y), X - Y \rangle + \frac{L}{2\tau} \| X - Y \|_F^2 + g(X).
\]
The KKT optimality conditions for the subproblem are given by
\begin{gather*}
\nabla f(Y) + \frac{L}{\tau}(X - Y) - \Diag(y) - \Lambda = 0, \\
\diag(X) = e, \quad X \succeq 0, \quad \Lambda \succeq 0, \quad \langle  \Lambda, X \rangle = 0.
\end{gather*}
The dual objective function of the subproblem is, up to an additive constant and a change in sign, given by
\[
\phi(y) :=
\frac{L}{2\tau} \left\| \left[ Y - \frac{\tau}{L} ( \nabla f(Y) - \Diag(y) ) \right]_+ \right\|_F^2  - \langle e, y \rangle.
\]
Note that $h$ is a differentiable convex function with gradient
\[
\nabla \phi(y) =
\diag \left(
\left[ Y - \frac{\tau}{L} ( \nabla f(Y) - \Diag(y) ) \right]_+
\right) - e.
\]
Suppose that $y$ solves
\[
\min_{y \in \mathbb{R}^n} \phi(y).
\]
Then $\nabla \phi(y) = 0$. We define $M$, $X$, and $\Lambda$ by
\[
M := Y - \frac{\tau}{L} ( \nabla f(Y) - \Diag(y) ),
\qquad
X := M_+,
\qquad
\Lambda := \frac{L}{\tau}(X - M) = -\frac{L}{\tau}M_-,
\]
where $M_+$ and $M_-$ are the projections of $M$ onto the set of positive semidefinite and negative semidefinite matrices, respectively. Note that $M = M_+ + M_-$ and $\langle M_+, M_- \rangle = 0$ by the Moreau decomposition theorem.
Thus we have $X \succeq 0$, $\diag(X) = e$, $\Lambda \succeq 0$, and $\langle \Lambda, X \rangle = 0$. Moreover,
\[
\Lambda
= \frac{L}{\tau}(X - M)
= \frac{L}{\tau}(X - Y) +  \nabla f(Y) - \Diag(y),
\]
which implies that
\[
\nabla f(Y) + \frac{L}{\tau}(X - Y) - \Diag(y) - \Lambda = 0.
\]
Thus, by minimizing the function $h$, we obtain the optimal solution of the subproblem.
Furthermore, letting
\[
\Gamma := -\Diag(y) - \Lambda,
\]
we have $\Gamma \in \partial g(X)$. Indeed, if $Z \in C$, then
\begin{align*}
g(X) + \langle \Gamma, Z - X \rangle
&= \langle -\Diag(y) - \Lambda, Z - X \rangle \\
&=
- \langle y, \diag(Z) \rangle
+ \langle y, \diag(X) \rangle
- \langle \Lambda, Z \rangle
+ \langle \Lambda, X \rangle \\
&=
- \langle \Lambda, Z \rangle \le 0 = g(Z),
\end{align*}
and if $Z \not\in C$, then $g(Z) = +\infty$, so $g(Z) \ge g(X) + \langle \Gamma, Z - X \rangle$ as well. Thus, we have shown that
\[
0 \in \nabla f(Y) + \frac{L}{\tau}(X - Y) + \partial g(X).
\]

\subsection{Approximately solving the subproblem}

In our implementation, we approximately minimize $\phi(y)$ using the quasi-Newton method L-BFGS-B \cite{Morales2011a,Zhu1997a}. Thus, we compute $y$ such that $\nabla \phi(y) \approx 0$, implying that $\diag(X) \approx e$. Thus, we expect that $X \not\in C$ and $g(X) = +\infty$. In order to satisfy the requirement that we have an $\varepsilon$-subgradient, it is necessary to have a point $\hat{X} \in C$. As is done in \cite{Borsdorf2010a,Defeng-S-2012}, we define $d := \diag(X)$ and $D := \Diag(d)^{-1/2}$; since $d \approx e$, we have that $D \succ 0$. We then let
\[
\hat{X} := DXD.
\]
Since $X \succeq 0$ and $D \succ 0$, we have that $\hat{X} \succeq 0$; moreover, $\diag(\hat{X}) = e$, as required.
Next we let
\[
\varepsilon := \langle \Lambda, \hat{X} \rangle
\quad \text{and} \quad
V := \nabla f(Y) + \frac{L}{\tau}(\hat{X} - Y) + \Gamma = \frac{L}{\tau}(\hat{X} - X).
\]
Note that $\varepsilon \ge 0$ since $\Lambda$ and $\hat{X}$ are both positive semidefinite. We claim that $\Gamma \in \partial_\varepsilon g(\hat{X})$. As before, if $Z \not\in C$, then $g(Z) = +\infty$, so $g(Z) \ge g(\hat{X}) + \langle \Gamma, Z - \hat{X} \rangle - \varepsilon$ holds. If $Z \in C$, then
\begin{align*}
g(\hat{X}) + \langle \Gamma, Z - \hat{X} \rangle - \varepsilon
&= \langle -\Diag(y) - \Lambda, Z - \hat{X} \rangle - \langle \Lambda, \hat{X} \rangle \\
&=
- \langle y, \diag(Z) \rangle
+ \langle y, \diag(\hat{X}) \rangle
- \langle \Lambda, Z \rangle
+ \langle \Lambda, \hat{X} \rangle
- \langle \Lambda, \hat{X} \rangle \\
&=
- \langle \Lambda, Z \rangle \le 0 = g(Z).
\end{align*}
Therefore, we have
\[
V \in \nabla f(Y) + \frac{L}{\tau}(\hat{X} - Y) + \partial_\varepsilon g(\hat{X}).
\]

\subsection{Computing projections}

Minimizing $\phi(y)$ using a quasi-Newton method like L-BFGS-B requires us to evaluate $\phi(y)$ and its gradient $\nabla \phi(y)$ for each new candidate minimizer $y$.
Each time we evaluate $\phi(y)$ and $\nabla \phi(y)$, we compute the projections $M_+$ and $M_-$ in order to compute $X$ and $\Lambda$. We do this by computing the full eigenvalue decomposition of $M$ and obtain $M_+$ (resp.~$M_-$) by setting the negative (resp.~positive) eigenvalues of $M$ to zero. The choice of eigensolver is important since around 90\% of the computation time is spent computing the eigenvalue decomposition of $M$. In our implementation of I-FISTA, IE-FISTA, and IA-FISTA, we compute $M_+$ and $M_-$ using the LAPACK \cite{Anderson1999a} \texttt{dsyevd} eigensolver to compute all the eigenvalues and eigenvectors of $M$; see Borsdorf and Higham \cite{Borsdorf2010a} for more on choice of eigensolver for computing $M_+$ in a preconditioned Newton algorithm for the nearest correlation matrix problem.

\subsection{Random instances}

For our numerical tests, we generate random $n \times n$ correlation matrices $U$ by sampling uniformly from the set of correlation matrices using the extended onion method \cite{Lewandowski2009a}. We then generate $n \times n$ symmetric matrices $G$ and $H$ using the following Julia code, based on the parameters $\gamma, p \in [0, 1]$, where $\gamma$ controls the amount of noise in $G$ and $p$ controls the sparsity of $H$.

\begin{lstlisting}
# Generate symmetric matrix E with ones on diagonal and off-diagonal entries
# sampled uniformly from the interval [-1, 1].
Etmp = 2 * rand(n, n) .- 1
E = Symmetric(triu(Etmp, 1) + I)

# Matrix G is the convex combination of the matrices U and E, with G = U when
# γ = 0 and G = E when γ = 1.
Gtmp = (1 - γ) .* U .+ γ .* E
G = Symmetric(triu(Gtmp, 1) + I)

# Generate symmetric matrix H with ones on diagonal and off-diagonal entries are
# uniformly sampled from the interval [0, 1] with probability p and are zero
# with probability 1 - p.
Htmp = [rand() < p ? rand() : 0.0 for i = 1:n, j = 1:n]
H = Symmetric(triu(Htmp, 1) + I)
\end{lstlisting}

For all our tests, we use $p = 0.5$ and we consider $n = 100, 200, \ldots, 800$ and $\gamma = 0.1, 0.2, \ldots, 1.0$, generating a random instance for each combination of $n$ and $\gamma$, giving us a total of eighty test instances.

\subsection{Numerical tests}

As was done in \cite{Defeng-S-2012}, we obtain a good initial point that is used by all three methods by solving the nearest correlation matrix problem
\[
\min_{X \in C} \frac{1}{2}\|X - G\|_F^2
\]
using the \textsc{Matlab} code \texttt{CorNewton3.m} \cite{Qi2009b} which is based on the quadratically convergent semismooth Newton method in \cite{Qi2006a}.

We also use a similar stopping criterion as the one used in \cite{Defeng-S-2012}. We let $r_p$ and $r_d$ be the norm of the primal and dual equality constraints for problem~\cref{HNCM}; that is,
\[
r_p := \|\diag(\hat{X}) - e\|_2,
\quad
r_d := \|\nabla f(\hat{X}) - \Diag(y) - \Lambda\|_F.
\]
Note that we are guaranteed to have $r_p$ be approximately machine epsilon based on how $\hat{X}$ is computed. We stop each method when
\[
\max\{r_p, r_d\} \le \mathtt{tol}.
\]
In our tests, we use $\mathtt{tol} = 10^{-1}$ since we found that using a smaller value of $\mathtt{tol}$ results in significantly more function/gradient evaluations and much longer running times for all three methods, but does not alter the main conclusions we draw from our numerical tests.

\begin{figure}
\includegraphics[width=\textwidth]{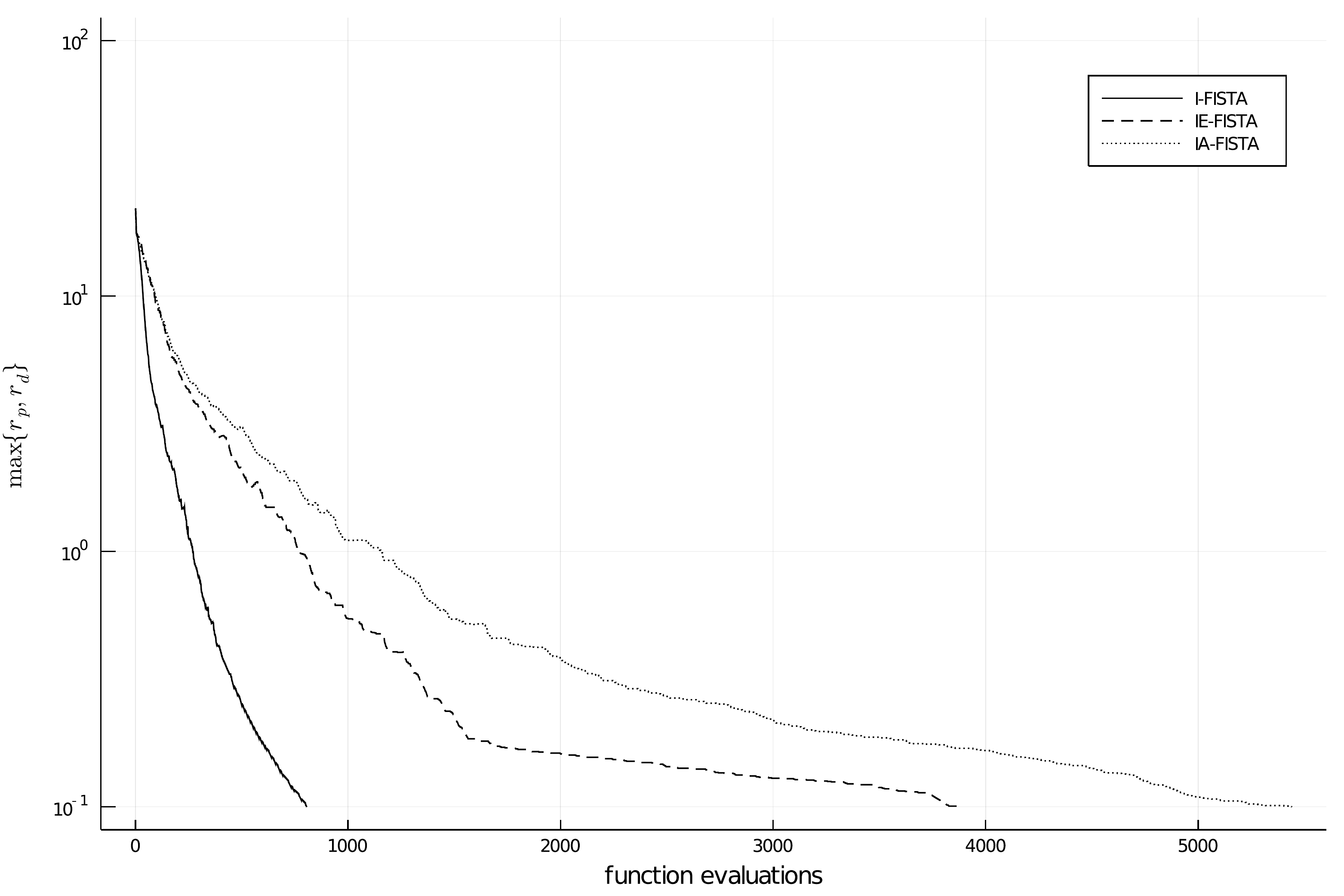}
\caption{Convergence plot for the \textup{I-FISTA}, \textup{IE-FISTA}, and \textup{IA-FISTA} methods on the $n = 400$ and $\gamma = 0.5$ test instance.}
\label{fig:convergence_plot}
\end{figure}

\begin{table}
  \centering
  \begin{tabular}{rr|rrr|rrr|rrr}
    \toprule
    $n$ & $\gamma$ & \multicolumn{3}{c|}{I-FISTA} & \multicolumn{3}{c|}{IE-FISTA} & \multicolumn{3}{c}{IA-FISTA} \\
    \cmidrule(lr){3-5} \cmidrule(lr){6-8} \cmidrule(lr){9-11}
    & & $k$ & fgs & time & $k$ & fgs & time & $k$ & fgs & time \\
    \midrule
100  &  0.10  &   27  &    45  &    0.1  &    35  &    155  &     0.3  &   26  &     52  &     0.1 \\
     &  0.20  &   56  &   102  &    0.2  &    70  &    229  &     0.4  &   53  &    160  &     0.4 \\
     &  0.30  &   88  &   169  &    0.4  &   111  &    360  &     0.6  &   85  &    330  &     0.7 \\
     &  0.40  &  131  &   261  &    0.4  &   164  &    600  &     1.1  &  127  &    624  &     1.3 \\
     &  0.50  &  130  &   269  &    0.4  &   162  &    617  &     1.1  &  125  &    676  &     1.3 \\
     &  0.60  &  141  &   296  &    0.5  &   176  &    709  &     1.2  &  137  &    784  &     1.4 \\
     &  0.70  &  143  &   300  &    0.5  &   178  &    827  &     1.4  &  139  &    817  &     1.4 \\
     &  0.80  &  144  &   304  &    1.1  &   175  &    860  &     2.1  &  140  &    857  &     2.7 \\
     &  0.90  &  147  &   310  &    0.6  &   179  &    871  &     1.9  &  142  &    887  &     1.8 \\
     &  1.00  &  148  &   320  &    1.0  &   179  &    729  &     2.2  &  143  &    922  &     2.6 \\
    \midrule
200  &  0.10  &   73  &   129  &    0.9  &    93  &    234  &     1.5  &   71  &    211  &     1.3 \\
     &  0.20  &  170  &   333  &    2.0  &   211  &    692  &     3.9  &  165  &    738  &     4.3 \\
     &  0.30  &  242  &   462  &    2.5  &   292  &   1099  &     5.9  &  235  &   1533  &     8.1 \\
     &  0.40  &  248  &   475  &    2.5  &   306  &   1201  &     6.3  &  239  &   1508  &     8.0 \\
     &  0.50  &  252  &   478  &    2.6  &   307  &   1165  &     6.1  &  243  &   1741  &     9.2 \\
     &  0.60  &  251  &   477  &    2.8  &   309  &   1207  &     6.5  &  244  &   1738  &     9.3 \\
     &  0.70  &  260  &   498  &    2.7  &   311  &   1122  &     6.0  &  250  &   1816  &     9.8 \\
     &  0.80  &  266  &   505  &    2.8  &   322  &   1422  &     7.5  &  257  &   1977  &    10.6 \\
     &  0.90  &  271  &   512  &    2.6  &   322  &   1711  &     8.9  &  260  &   1813  &     9.6 \\
     &  1.00  &  273  &   520  &    2.7  &   327  &   1638  &     8.5  &  263  &   2044  &    10.7 \\
    \midrule
300  &  0.10  &  124  &   241  &    2.9  &   147  &    403  &     5.1  &  121  &    414  &     5.4 \\
     &  0.20  &  330  &   620  &    7.6  &   369  &   1421  &    17.5  &  312  &   2104  &    25.4 \\
     &  0.30  &  332  &   621  &    7.7  &   383  &   1573  &    19.6  &  322  &   2699  &    34.4 \\
     &  0.40  &  347  &   662  &    8.1  &   395  &   1447  &    17.8  &  335  &   3127  &    38.3 \\
     &  0.50  &  355  &   640  &    8.1  &   405  &   1825  &    23.3  &  343  &   3136  &    39.3 \\
     &  0.60  &  366  &   644  &    8.1  &   419  &   2858  &    36.4  &  355  &   3361  &    41.5 \\
     &  0.70  &  371  &   648  &    8.0  &   423  &   3674  &    45.5  &  359  &   4013  &    48.8 \\
     &  0.80  &  376  &   685  &    8.7  &   427  &   3672  &    46.7  &  364  &   4156  &    52.2 \\
     &  0.90  &  380  &   703  &    8.8  &   430  &   4022  &    51.6  &  366  &   4855  &    59.7 \\
     &  1.00  &  388  &   734  &    9.3  &   432  &   4223  &    53.7  &  373  &   4742  &    59.3 \\
    \midrule
400  &  0.10  &  182  &   333  &    7.6  &   233  &    450  &    10.5  &  177  &    737  &    17.1 \\
     &  0.20  &  413  &   783  &   17.4  &   511  &   2050  &    45.9  &  401  &   3829  &    82.2 \\
     &  0.30  &  431  &   778  &   17.4  &   531  &   1790  &    40.1  &  418  &   5430  &   118.4 \\
     &  0.40  &  450  &   803  &   18.4  &   549  &   2503  &    56.1  &  436  &   5147  &   111.8 \\
     &  0.50  &  467  &   806  &   18.0  &   567  &   3715  &    83.9  &  453  &   5440  &   118.1 \\
     &  0.60  &  479  &   881  &   19.7  &   578  &   4632  &   103.9  &  462  &   5828  &   126.8 \\
     &  0.70  &  489  &   919  &   20.6  &   585  &   4075  &    91.3  &  472  &   6439  &   140.8 \\
     &  0.80  &  499  &   945  &   21.2  &   596  &   3701  &    83.3  &  480  &   6974  &   153.2 \\
     &  0.90  &  507  &   947  &   21.3  &   602  &   3944  &    89.1  &  488  &   7543  &   165.1 \\
     &  1.00  &  509  &   971  &   22.2  &   611  &   3459  &    78.8  &  491  &   7490  &   166.2 \\
    \bottomrule
  \end{tabular}
  \caption{The number of iterations ($k$), function/gradient evaluations (fgs), and time in seconds for the \textup{I-FISTA}, \textup{IE-FISTA}, and \textup{IA-FISTA} methods for $n = 100, 200, 300, 400$.}
  \label{table1}
\end{table}

\begin{table}
  \centering
  \begin{tabular}{rr|rrr|rrr|rrr}
    \toprule
    $n$ & $\gamma$ & \multicolumn{3}{c|}{I-FISTA} & \multicolumn{3}{c|}{IE-FISTA} & \multicolumn{3}{c}{IA-FISTA} \\
    \cmidrule(lr){3-5} \cmidrule(lr){6-8} \cmidrule(lr){9-11}
    & & $k$ & fgs & time & $k$ & fgs & time & $k$ & fgs & time \\
    \midrule
500  &  0.10  &  265  &   513  &   20.2  &   320  &    667  &    26.3  &  257  &   1203  &    45.8 \\
     &  0.20  &  499  &   934  &   35.0  &   595  &   2429  &    94.6  &  484  &   5940  &   216.7 \\
     &  0.30  &  529  &   924  &   35.3  &   626  &   2716  &   102.8  &  513  &   7012  &   261.6 \\
     &  0.40  &  559  &   994  &   37.4  &   654  &   5697  &   214.6  &  539  &   9331  &   340.8 \\
     &  0.50  &  573  &  1072  &   40.5  &   669  &   4672  &   175.9  &  553  &   7133  &   260.5 \\
     &  0.60  &  591  &  1117  &   42.4  &   679  &   4781  &   182.3  &  569  &   7549  &   279.0 \\
     &  0.70  &  598  &  1093  &   41.5  &   700  &   6730  &   256.2  &  578  &   7455  &   275.4 \\
     &  0.80  &  607  &  1173  &   44.5  &   705  &   5664  &   215.5  &  585  &   8618  &   318.7 \\
     &  0.90  &  617  &  1146  &   43.6  &   711  &   4724  &   180.1  &  594  &   8562  &   318.1 \\
     &  1.00  &  626  &  1227  &   46.7  &   721  &   3796  &   144.9  &  604  &   9962  &   369.4 \\
    \midrule
600  &  0.10  &  377  &   709  &   46.6  &   434  &   1102  &    72.0  &  366  &   2919  &   185.3 \\
     &  0.20  &  586  &  1085  &   68.4  &   678  &   3109  &   194.5  &  567  &   9471  &   576.4 \\
     &  0.30  &  629  &  1089  &   67.2  &   723  &   4896  &   303.7  &  610  &   8744  &   525.9 \\
     &  0.40  &  662  &  1132  &   66.8  &   750  &   6461  &   383.7  &  636  &   9164  &   526.4 \\
     &  0.50  &  683  &  1263  &   76.1  &   777  &   5335  &   318.3  &  660  &  10852  &   629.3 \\
     &  0.60  &  696  &  1282  &   78.5  &   792  &   5446  &   332.7  &  673  &  12647  &   751.9 \\
     &  0.70  &  710  &  1333  &   82.5  &   801  &   7079  &   439.5  &  685  &  12607  &   758.0 \\
     &  0.80  &  729  &  1368  &   82.9  &   811  &   5087  &   307.4  &  703  &   9818  &   576.0 \\
     &  0.90  &  733  &  1388  &   84.3  &   832  &   4203  &   257.1  &  708  &  10843  &   640.3 \\
     &  1.00  &  739  &  1439  &   89.0  &   840  &   5726  &   354.7  &  715  &  10698  &   645.6 \\
    \midrule
700  &  0.10  &  521  &   979  &   89.4  &   586  &   1859  &   169.5  &  498  &   6468  &   571.7 \\
     &  0.20  &  673  &  1251  &  109.0  &   762  &   3394  &   298.6  &  651  &  10581  &   896.6 \\
     &  0.30  &  725  &  1217  &  105.7  &   819  &   6842  &   596.5  &  703  &  12948  &  1094.2 \\
     &  0.40  &  762  &  1424  &  124.1  &   847  &   5549  &   485.3  &  733  &  11215  &   951.6 \\
     &  0.50  &  786  &  1380  &  121.6  &   871  &   6919  &   606.9  &  756  &  11574  &   987.7 \\
     &  0.60  &  803  &  1524  &  135.3  &   888  &   7032  &   626.5  &  773  &  13455  &  1160.8 \\
     &  0.70  &  819  &  1511  &  133.9  &   911  &   8809  &   783.1  &  791  &  15117  &  1301.1 \\
     &  0.80  &  828  &  1540  &  137.7  &   923  &   8837  &   792.4  &  800  &  15366  &  1335.4 \\
     &  0.90  &  842  &  1624  &  142.8  &   932  &   9432  &   829.4  &  813  &  15936  &  1369.9 \\
     &  1.00  &  855  &  1661  &  148.7  &   941  &   9716  &   873.2  &  824  &  17285  &  1509.8 \\
    \midrule
800  &  0.10  &  692  &  1286  &  165.3  &   756  &   3108  &   400.4  &  670  &  11408  &  1420.6 \\
     &  0.20  &  762  &  1421  &  177.2  &   848  &   3697  &   454.4  &  738  &  14087  &  1682.3 \\
     &  0.30  &  822  &  1347  &  173.1  &   907  &   6835  &   878.1  &  794  &  12437  &  1542.9 \\
     &  0.40  &  862  &  1573  &  199.0  &   947  &   7783  &   995.7  &  832  &  14689  &  2041.1 \\
     &  0.50  &  887  &  1639  &  208.1  &   970  &   7028  &   893.4  &  856  &  16092  &  1985.1 \\
     &  0.60  &  910  &  1700  &  213.7  &   995  &   9012  &  1131.0  &  879  &  18385  &  2242.0 \\
     &  0.70  &  923  &  1747  &  225.2  &  1013  &   9210  &  1188.9  &  892  &  18666  &  2296.6 \\
     &  0.80  &  943  &  1805  &  231.3  &  1023  &   8011  &  1032.9  &  905  &  13242  &  1651.0 \\
     &  0.90  &  960  &  1840  &  238.7  &  1036  &  10180  &  1324.2  &  927  &  14950  &  1892.3 \\
     &  1.00  &  971  &  1884  &  243.9  &  1056  &  11607  &  1502.1  &  937  &  16054  &  2026.2 \\
     \bottomrule
  \end{tabular}
  \caption{The number of iterations ($k$), function/gradient evaluations (fgs), and time in seconds for the \textup{I-FISTA}, \textup{IE-FISTA}, and \textup{IA-FISTA} methods for $n = 500, 600, 700, 800$.}
  \label{table2}
\end{table}

\begin{figure}
\includegraphics[width=\textwidth]{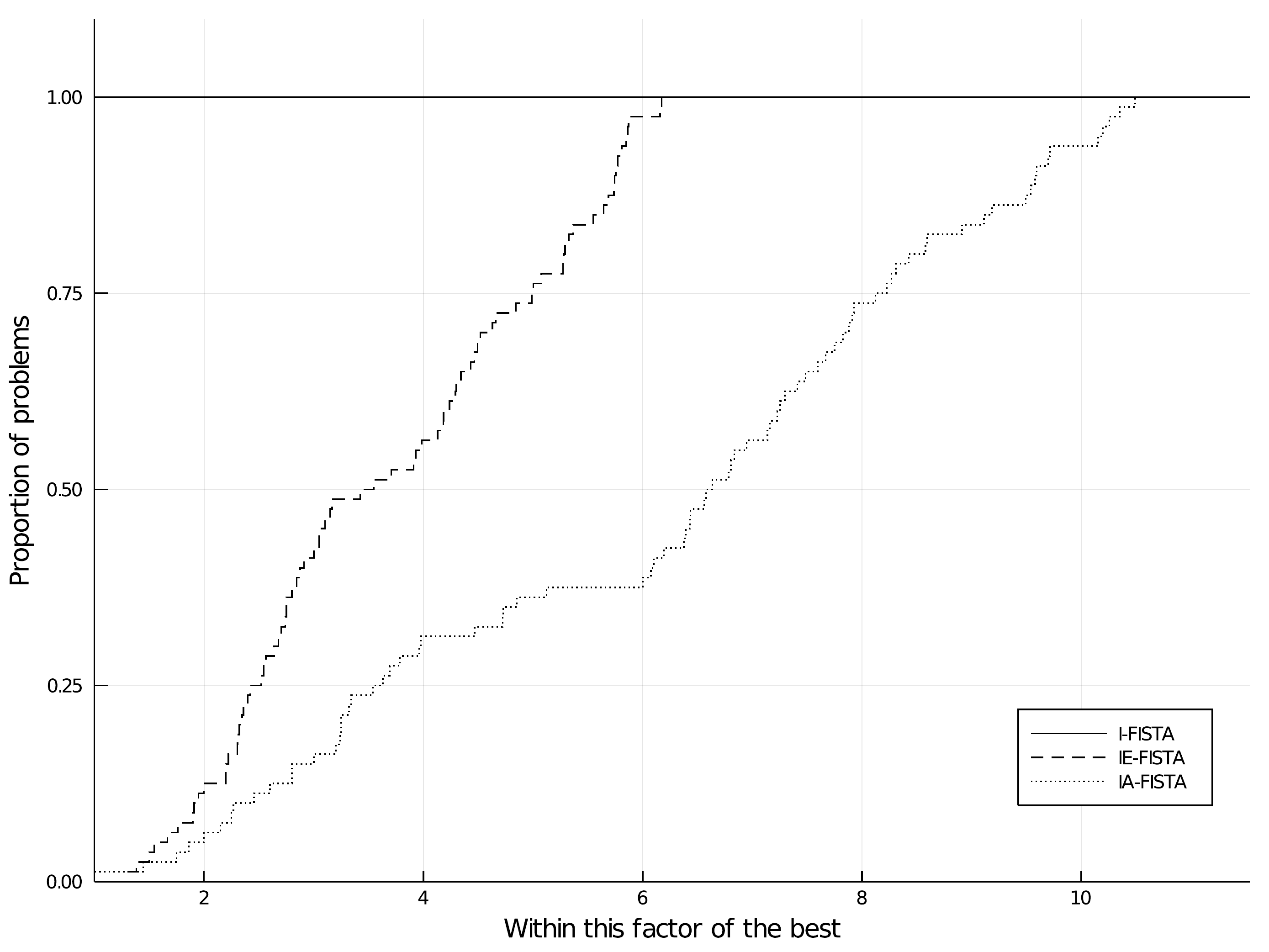}
\caption{Performance profile of total running time for the \textup{I-FISTA}, \textup{IE-FISTA}, and \textup{IA-FISTA} methods on all test instances.}
\label{fig:performance_profile}
\end{figure}

An example of the typical convergence behavior of the three methods is shown in \cref{fig:convergence_plot} where the value of $\max\{r_p,r_d\}$ is plotted each time $\phi(y)$ and $\nabla \phi(y)$ are evaluated. Note, however, that during the linesearch procedure of L-BFGS-B, the value of $\max\{r_p,r_d\}$ may vary drastically, so, to obtain a plot without such  noise, we replace those intermediate linesearch values with the value obtained at the termination of the linesearch or when the stopping condition for the subproblem is satisfied.

In \cref{table1,table2} we record the number of outer iterations ($k$), the number of function/gradient evaluations (fgs), and the total running time in seconds, but not including the time to compute the initial point. From these results, it is clear that $k$, fgs, and time increase for all three methods as $n$ increases and as $\gamma$ increases. However, we also see that although I-FISTA and IE-FISTA require more outer iterations than IA-FISTA, each require fewer total inner iterations (i.e., fgs), and hence less time, than IA-FISTA to solve each instance to the desired tolerance.

Here we include an interesting point. In our numerical tests we observed that L-BFGS-B was always able to satisfy the IR Rule, typically in a small number of iterations. However, we were curious to see that sometimes L-BFGS-B failed to satisfy the IER Rule and only stopped due to a failure of the linesearch or due to having identical function values on two consecutive function evaluations. We would like to investigate this behavior in greater detail in future research.

To see the forest for the trees, in \cref{fig:performance_profile} we plot the performance profile \cite{Dolan2002a,Gould2016a,More2009a} of the numerical results from \cref{table1,table2} using the total running time of each solver on each instance. From this plot we clearly see that I-FISTA is the fastest on all test instances and that IE-FISTA also outperforms IA-FISTA on the test instances. Thus, although all three algorithms have the same theoretical rate of convergence, we have demonstrated that the relative error rules and corresponding algorithms proposed in this paper, I-FISTA and, to a lesser extent, IE-FISTA, are potentially valuable to use in situations where IA-FISTA has proved successful in practice.

\section{Final Remarks}\label{Finalremarks}

This paper proposed and analyzed two inexact versions of FISTA for minimizing the sum of two convex functions.
Both schemes allow their subproblems to be solved inexactly subject to satisfying certain relative error rules.
Numerical experiments were carried out in order to illustrate the numerical behavior of the methods.
They indicate that the proposed methods based on  inexact relative error rules are more efficient than those based on the inexact absolute error rule on a set of instances of the $H$-weighted nearest correlation matrix problem.

\bibliographystyle{siamplain}
\bibliography{refs}
\end{document}

%% file: Teste_shared.tex
\usepackage{lipsum}
\usepackage{amsfonts}
\usepackage{graphicx}
\usepackage{epstopdf}
\usepackage{algorithmic}

\ifpdf
\DeclareGraphicsExtensions{.eps,.pdf,.png,.jpg}
\else
\DeclareGraphicsExtensions{.eps}
\fi

\newcommand{\TheTitle}{On Inexact Accelerated Proximal Gradient Methods with Relative Error Rules}

\headers{Inexact Accelerated Proximal Gradient Methods}{Y. Bello-Cruz,  M. L. N. Gon\c{c}alves, and N. Krislock}

\title{{\TheTitle}}

\author{
    Yunier Bello-Cruz
    \thanks{Department of Mathematical Sciences, Northern Illinois University, DeKalb, IL 60115, USA. (E-mails: {\tt
       yunierbello@niu.edu} and {\tt nkrislock@niu.edu}). YBC was partially supported by the \emph{National Science Foundation} (NSF), Grant DMS -- 1816449.}
   \and
   Max L. N. Gon\c calves
    \thanks{IME, Universidade Federal de Goi\'as,  Goi\^ania-GO 74001-970, Brazil. (E-mail: {\tt
       maxlng@ufg.br}).     MLNG was partially supported by the \emph{Brazilian Agency  Conselho Nacional de Pesquisa} (CNPq), Grants 302666/2017-6 and 408123/2018-4.  }
     \and
      Nathan Krislock \footnotemark[2]
    }

\usepackage{amsopn}
\DeclareMathOperator{\diag}{diag}
\DeclareMathOperator{\Diag}{Diag}
\DeclareMathOperator{\trace}{trace}


%% file: Inexact-Fista.bbl
\def\cprime{$'$}
\begin{thebibliography}{10}

\bibitem{Anderson1999a}
{\sc E.~Anderson, Z.~Bai, C.~Bischof, S.~Blackford, J.~Demmel, J.~Dongarra,
  J.~Du~Croz, A.~Greenbaum, S.~Hammarling, A.~McKenney, and D.~Sorensen}, {\em
  {LAPACK} Users' Guide}, Society for Industrial and Applied Mathematics,
  Philadelphia, PA, third~ed., 1999.

\bibitem{Attouch-Cabot-HAL2017}
{\sc H.~Attouch and A.~Cabot}, {\em Convergence rates of inertial
  forward-backward algorithms}, SIAM J. Optim., 28 (2018), pp.~849--874.

\bibitem{Attouch-JOTA-18}
{\sc H.~Attouch, A.~Cabot, Z.~Chbani, and H.~Riahi}, {\em Inertial
  forward-backward algorithms with perturbations: application to {T}ikhonov
  regularization}, J. Optim. Theory Appl., 179 (2018), pp.~1--36.

\bibitem{Attouch-fast-MPB-2018}
{\sc H.~Attouch, Z.~Chbani, J.~Peypouquet, and P.~Redont}, {\em Fast
  convergence of inertial dynamics and algorithms with asymptotic vanishing
  viscosity}, Math. Program., 168 (2018), pp.~123--175.

\bibitem{Attouch-rate-2016}
{\sc H.~Attouch and J.~Peypouquet}, {\em The rate of convergence of
  {N}esterov's accelerated forward-backward method is actually faster than
  {$1/k^2$}}, SIAM J. Optim., 26 (2016), pp.~1824--1834.

\bibitem{Aujol-Dossal-2015}
{\sc J.-F. Aujol and C.~Dossal}, {\em Stability of over-relaxations for the
  forward-backward algorithm, application to {FISTA}}, SIAM J. Optim., 25
  (2015), pp.~2408--2433.

\bibitem{B-Bu-2019}
{\sc H.~H. Bauschke, M.~Bui, and X.~Wang}, {\em Applying {FISTA} to
  optimization problems (with or) without minimizers}, Mathematical
  Programming, 192 (2019), pp.~1--20.

\bibitem{FISTA}
{\sc A.~Beck and M.~Teboulle}, {\em A fast iterative shrinkage-thresholding
  algorithm for linear inverse problems}, SIAM Journal on Imaging Sciences, 2
  (2009), pp.~183--202.

\bibitem{BTe}
{\sc A.~Beck and M.~Teboulle}, {\em Gradient-based algorithms with applications
  to signal-recovery problems}, in Convex optimization in signal processing and
  communications, Cambridge Univ. Press, Cambridge, 2010, pp.~42--88.

\bibitem{FB-2014}
{\sc J.~Y. Bello~Cruz}, {\em On proximal subgradient splitting method for
  minimizing the sum of two nonsmooth convex functions}, Set-Valued and
  Variational Analysis, 25 (2017), pp.~245--263.

\bibitem{YN2}
{\sc J.~Y. Bello~Cruz, G.~Li, and T.~T.~A. Nghia}, {\em On the {$Q$}-linear
  convergence of forward-backward splitting method and uniqueness of optimal
  solution to {Lasso}}, 2018, \href{http://arxiv.org/abs/1806.06333}
  {arXiv:1806.06333}.

\bibitem{YN}
{\sc J.~Y. Bello~Cruz and T.~A. Nghia}, {\em On the convergence of the
  forward--backward splitting method with linesearches}, Optim. Methods Softw.,
  31 (2016), pp.~1209--1238.

\bibitem{Bezanson2017a}
{\sc J.~Bezanson, A.~Edelman, S.~Karpinski, and V.~Shah}, {\em Julia: A fresh
  approach to numerical computing}, SIAM Review, 59 (2017), pp.~65--98.

\bibitem{Borsdorf2010a}
{\sc R.~Borsdorf and N.~J. Higham}, {\em {A preconditioned Newton algorithm for
  the nearest correlation matrix}}, IMA Journal of Numerical Analysis, 30
  (2010), pp.~94--107.

\bibitem{Chambolle-Dossal-15}
{\sc A.~Chambolle and C.~Dossal}, {\em On the convergence of the iterates of
  the ``fast iterative shrinkage/thresholding algorithm''}, J. Optim. Theory
  Appl., 166 (2015), pp.~968--982.

\bibitem{Dolan2002a}
{\sc E.~D. Dolan and J.~J. Mor{\'e}}, {\em Benchmarking optimization software
  with performance profiles}, Mathematical Programming, 91 (2002),
  pp.~201--213.

\bibitem{Gould2016a}
{\sc N.~Gould and J.~Scott}, {\em A note on performance profiles for
  benchmarking software}, ACM Trans. Math. Softw., 43 (2016).

\bibitem{HYZ}
{\sc E.~T. Hale, W.~Yin, and Y.~Zhang}, {\em Fixed-point continuation for
  {$l_1$}-minimization: methodology and convergence}, SIAM J. Optim., 19
  (2008), pp.~1107--1130.

\bibitem{Defeng-S-2012}
{\sc K.~Jiang, D.~Sun, and K.-C. Toh}, {\em An inexact accelerated proximal
  gradient method for large scale linearly constrained convex {SDP}}, SIAM J.
  Optim., 22 (2012), pp.~1042--1064.

\bibitem{Lewandowski2009a}
{\sc D.~Lewandowski, D.~Kurowicka, and H.~Joe}, {\em Generating random
  correlation matrices based on vines and extended onion method}, Journal of
  Multivariate Analysis, 100 (2009), pp.~1989 -- 2001.

\bibitem{Complexity-HPE}
{\sc R.~D.~C. Monteiro and B.~F. Svaiter}, {\em On the complexity of the hybrid
  proximal extragradient method for the iterates and the ergodic mean}, SIAM
  Journal on Optimization, 20 (2010), pp.~2755--2787.

\bibitem{acc-HPE}
{\sc R.~D.~C. Monteiro and B.~F. Svaiter}, {\em An accelerated hybrid proximal
  extragradient method for convex optimization and its implications to
  second-order methods}, SIAM Journal on Optimization, 23 (2013),
  pp.~1092--1125.

\bibitem{Morales2011a}
{\sc J.~L. Morales and J.~Nocedal}, {\em Remark on ``{Algorithm} 778:
  {L-BFGS-B}: Fortran subroutines for large-scale bound constrained
  optimization''}, ACM Trans. Math. Softw., 38 (2011), pp.~1--4.

\bibitem{More2009a}
{\sc J.~J. Mor\'{e} and S.~M. Wild}, {\em Benchmarking derivative-free
  optimization algorithms}, SIAM Journal on Optimization, 20 (2009),
  pp.~172--191.

\bibitem{Nesterov1983}
{\sc Y.~Nesterov}, {\em A method for solving the convex programming problem
  with convergence rate {$O(1/k^{2})$}}, Dokl. Akad. Nauk SSSR, 269 (1983),
  pp.~543--547.

\bibitem{Nesterov1988}
{\sc Y.~Nesterov}, {\em An approach to constructing optimal methods for
  minimization of smooth convex functions}, \`Ekonom. i Mat. Metody, 24 (1988),
  pp.~509--517.

\bibitem{Nesterov2005}
{\sc Y.~Nesterov}, {\em Smooth minimization of non-smooth functions}, Math.
  Program., 103 (2005), pp.~127--152.

\bibitem{Nesterov2013}
{\sc Y.~Nesterov}, {\em Gradient methods for minimizing composite functions},
  Mathematical Programming, 140 (2013), pp.~125--161.

\bibitem{Qi2006a}
{\sc H.~Qi and D.~Sun}, {\em A quadratically convergent {N}ewton method for
  computing the nearest correlation matrix}, SIAM J. Matrix Anal. Appl., 28
  (2006), pp.~360--385.

\bibitem{Qi2009b}
{\sc H.~Qi, D.~Sun, and Y.~Gao}, {\em \texttt{CorNewton3.m}: A \textsc{Matlab}
  code for computing the nearest correlation matrix with fixed diagonal and off
  diagonal elements}.
\newblock https://www.polyu.edu.hk/ama/profile/dfsun/CorNewton3.m, 2009.

\bibitem{HPE2}
{\sc M.~V. Solodov and B.~F. Svaiter}, {\em A hybrid approximate
  extragradient-proximal point algorithm using the enlargement of a maximal
  monotone operator}, Set-Valued Anal., 7 (1999), pp.~323--345.

\bibitem{Su-Boyd-Candes-2016}
{\sc W.~Su, S.~Boyd, and E.~J. Cand\`es}, {\em A differential equation for
  modeling {N}esterov's accelerated gradient method: theory and insights}, J.
  Mach. Learn. Res., 17 (2016), pp.~Paper No. 153, 43.

\bibitem{Tr}
{\sc J.~A. Tropp}, {\em Just relax: convex programming methods for identifying
  sparse signals in noise}, IEEE Trans. Inform. Theory, 52 (2006),
  pp.~1030--1051.

\bibitem{Villa-Salzo-Luca-Verri-2013}
{\sc S.~Villa, S.~Salzo, L.~Baldassarre, and A.~Verri}, {\em Accelerated and
  inexact forward-backward algorithms}, SIAM J. Optim., 23 (2013),
  pp.~1607--1633.

\bibitem{Zhu1997a}
{\sc C.~Zhu, R.~H. Byrd, P.~Lu, and J.~Nocedal}, {\em Algorithm 778:
  {L-BFGS-B}: Fortran subroutines for large-scale bound-constrained
  optimization}, ACM Trans. Math. Softw., 23 (1997), pp.~550--560.

\end{thebibliography}
